\documentclass[11pt,psamsfonts]{amsart}
\usepackage{amsmath}
\usepackage{amsthm}
\usepackage{amssymb}
\usepackage{amscd}
\usepackage{amsfonts}
\usepackage{amsbsy}
\usepackage{graphicx}
\usepackage[dvips]{psfrag}
\usepackage{array}
\usepackage{color}
\usepackage{epsfig,epstopdf}
\usepackage{url}
\usepackage{color}
\usepackage{CJK}

\usepackage{float}
\usepackage{tikz}
\usepackage{subcaption}

\newtheorem {theorem} {Theorem}
\newtheorem {proposition} [theorem]{Proposition}

\newtheorem {lemma}  [theorem]{Lemma}

\title[Planar quasi--homogeneous differential systems]
{GLOBAL DYNAMICS OF QUADRATIC AND CUBIC
PLANAR QUASI-HOMOGENEOUS DIFFERENTIAL
SYSTEMS}

\author[J. Llibre, Y. Tang, J. Yu, P. Zhou]
{Jaume Llibre, Yilei Tang, Jiang Yu, Pengyu Zhou}

\email{jllibre@mat.uab.cat(J. Llibre), mathtyl@sjtu.edu.cn (Y. Tang), jiangyu@sjtu.edu.cn(J. Yu), e-c-h-o@sjtu.edu.cn(P. Zhou)}

\keywords{Quasi-homogeneous differential system; global dynamics; blow-up; Poincaré compactification; normal sector method.}

\begin{document}

\begin{abstract}
 In this paper we obtain the global dynamics and phase portraits of quadratic and cubic quasi-homogeneous but non-homogeneous systems. We first prove that all planar quadratic and cubic quasi-homogeneous but non-homogeneous polynomial systems can be reduced to three homogeneous ones. Then for the homogeneous systems, we employ blow-up method, normal sector method, Poincaré compactification and other techniques to discuss their dynamics. Finally we characterize the global phase portraits of quadratic and cubic quasi-homogeneous but non-homogeneous polynomial systems.
\end{abstract}

\maketitle

\section{ Introduction }\label{s1}

Planar differential systems serve as a crucial tool for characterizing the dynamic relationships between two interacting entities. In disciplines such as automatic control theory, aerospace technology, biological sciences, and economics, the qualitative theory of planar differential systems is an indispensable mathematical tool. This paper will focus on the global structures of quasi-homogeneous differential system, which is a generalization of the homogeneous differential system.

Consider a real planar polynomial differential system
\begin{equation}\label{eqn-1}
	\dot{x}=P(x,y),\quad\dot{y}=Q(x,y),
\end{equation}
where $P(x,y)$, $Q(x,y)\in \mathbb{R}[x,y], PQ\neq 0.$ As usual, the dot denotes derivative with respect to an independent real variable t and $\mathbb{R}[x,y]$ denotes the ring of polynomials in the variables $x$ and $y$ with coefficiants in $\mathbb{R}$. We say that system \eqref{eqn-1} has degree $n$ if $n=\max\{\deg P,\deg Q\}$. In what follows we assume without loss of generality that $P$ and $Q$ in system \eqref{eqn-1} have not a non-constant common factor.

If $P(x_0, y_0)=Q(x_0, y_0)=0$, then $(x_0, y_0)$ is a singularity of the system \eqref{eqn-1}. Without loss of generality, we can move a singularity $(x_0, y_0)$ to the origin.

System \eqref{eqn-1} is called a \textit{quasi-homogeneous polynomial differential system} if there exists constants $s_1, s_2, d\in \mathbb{N}$ such that for an arbitrary $\alpha \in \mathbb{R_+}$ it holds that
\begin{equation}\label{eqn-2}
	P(\alpha^{s_1}x,\alpha^{s_2}y)=\alpha^{s_1+d-1}P(x,y),\quad Q(\alpha^{s_1}x,\alpha^{s_2}y)=\alpha^{s_2+d-1}Q(x,y),
\end{equation}
where $\mathbb{N}$ is the set of positive integers and $\mathbb{R_+}$ is the set of positive real numbers. We call $(s_1, s_2)$ \textit{weight exponents} of system \eqref{eqn-1} and $d$ \textit{weight degree} with respect to the weight exponents. Moreover, $w=(s_1, s_2, d)$ is a denominated \textit{weight vector} of system \eqref{eqn-1} or of its associated vector field. For a quasi-homogeneous polynomial differential system \eqref{eqn-1}, a weight vector $\widetilde{w}=(\widetilde{s}_1,\widetilde{s}_2,\tilde{d})$ is minimal for system \eqref{eqn-1} if any other weight vector $(s_1, s_2, d)$ of system \eqref{eqn-1} satisfies $\tilde{s}_1\leq s_1,\tilde{s}_2\leq s_2\mathrm{~and~}\tilde{d}\leq d$. When $s_1=s_2=1$, system \eqref{eqn-1} is a homogeneous one of degree $d$. In our following discussion, we consider the system as quasi-homogeneous but non-homogeneous polynomial systems.

We write the polynomials $P$ and $Q$ of system \eqref{eqn-1} in its homogeneous parts
\begin{equation}\label{eqn-3}
	P(x,y)=\sum_{j=0}^{l}P_{j}(x,y),\quad\mathrm{where~}P_{j}(x,y)=\sum_{i=0}^{j}a_{i,j-i}x^{i}y^{j-i},
\end{equation}
and
\begin{equation}\label{eqn-4}
	Q(x,y)=\sum_{j=0}^{l}Q_{j}(x,y),\quad\mathrm{where~}Q_{j}(x,y)=\sum_{i=0}^{j}b_{i,j-i}x^{i}y^{j-i},
\end{equation}

Then from  García et al. \cite{GARCIA20133185} we know that $(0,0)$ the unique finite singular point of system \eqref{eqn-1}.

In 1900, at the International Congress of Mathematicians, Hilbert proposed 23 famous mathematical problems. The second part of the 16th problem concerns the lower upper bound of limit cycles that an $n$th-degree planar polynomial differential system can have—denoted as $H(n)$, also known as the Hilbert number. While it is known that any given planar polynomial differential system has a finite number of limit cycles, no substantial results have been obtained for $H(n)$ when $ n > 1$. This problem lies at the intersection of algebraic geometry, topology, and the theory of dynamical systems, and continues to stimulate mathematical research more than a century after it was first posed.

Homogeneous differential systems are a class of special polynomial systems, which have been intensively investigated by many authors, see e.g. \cite{Sibirskii1977, Date1979, Vdovina1984,Newton1978} for quadratic homogeneous systems, \cite{Cima1990, Ye1995} for cubic homogeneous systems, and \cite{Cima1990, Llibre1996} for homogeneous systems of arbitrary degree. These papers either provide a characterization of the phase portraits of homogeneous polynomial vector fields of degrees 2 and 3, offer an algebraic classification of homogeneous vector fields, or describe the structurally stable homogeneous vector fields.

With the growing complexity of real-world systems, quasi-homogeneous polynomial systems have gained increased attention due to their structural generality and analytical flexibility. These systems, which generalize homogeneous ones by allowing different weights for each variable, provide a more nuanced framework for capturing dynamics in nonlinear contexts. Moreover, for a polynomial differential system, typically we write it in terms of its homogeneous monomials. When we use classical blow-up to characterize the degenerate singularities, maybe many times of blow-up are needed. However, when we use specific techniques, such as quasi-homogeneous blow-up to characterize degenerate singularities, writing the system in its quasi-homogeneous terms turned out to be more effective.

In 2005, Zhao, Zhang \cite{ZHAO2005563} considered $H$ as quasi-homogeneous polynomial system. It is well known that one of the primary methods for generating limit cycles is perturbing a system with a center. In the case of planar quasi-homogeneous polynomial systems, these systems exhibit certain distinctive properties, such as possessing a global center and having a period function with monotonicity, see e.g. \cite{Zhang2018}. Consequently, the bifurcation of limit cycles from quasi-homogeneous polynomial centers has attracted significant attention in the field of qualitative theory of differential equations, see e.g. \cite{Li2009}. These systems serve as a bridge between homogeneous and general polynomial systems, offering a more tractable yet nontrivial class for both theoretical analysis and computational studies. The quasi-homogeneous polynomial systems have also been researched from different point of view such as integrability, see \cite{Algaba2009, Gine2013, Goriely1996, Hu2007} and their references. all planar quasi-homogeneous vector fields are Liouvillian integrable, see e.g. \cite{Garcia2003, Garcia2013, Li2009}.

In the 21st century, the study of quasi-homogeneous systems has emerged as a prominent research topic. For example, García et al. \cite{GARCIA20133185} investigated the integrability of planar quasi-homogeneous polynomial differential systems, while Aziz et al. \cite{AZIZ2014233} explore the problems of quasi-homogeneous centers and limit cycle bifurcations. García et al. \cite{GARCIA20133185} also proposed an algorithm for obtaining the explicit expressions of all quasi-homogeneous but non-homogeneous systems and applies this algorithm to derive the canonical forms of all second and third degree quasi-homogeneous polynomial systems.

Building on this foundation, some researchers have further applied the algorithm to obtain the standard forms of quasi-homogeneous but non-homogeneous systems and analyze their dynamical behavior. For instance, Tang and Zhang \cite{TANG201990} explored fifth-degree quasi-homogeneous polynomial systems, and based on this, Martínez et al. \cite{MARTINEZ20165923} conducted a systematic analysis, providing their standard forms and determines the first integrals and global phase portraits of all fifth degree coprime systems. These comprehensive classifications have significantly advanced our understanding of the structure and complexity of quasi-homogeneous differential systems and paved the way for further exploration of higher-degree and more generalized system classes.

Another crucial aspect of the qualitative theory of planar differential systems is the study of global phase portraits. Since global phase portraits provide a clear and comprehensive depiction of a system's global dynamical behavior, they are of great significance in both theoretical research and practical applications. Understanding the global structure helps in classifying different types of dynamical behaviors, revealing symmetries, invariant sets, and bifurcation scenarios that cannot be detected through local analysis alone. Therefore, global phase portrait analysis has become an indispensable step in the full qualitative classification of differential systems.



\bigskip


\section{Preliminaries} \label{s2}

In this section we will introduce the primary analytical tools employed in this study, including the blow-up technique, the \text{Poincar\'{e}}
compactification, and relevant techniques for qualitative analysis, such as normal sector method.

Blow-up is the basic tool for studying non-elementary singularities of a differential system in the plane, based on changes of variables. We use this technique for classifying the nilpotent singularities; i.e., the singularities having both eigenvalues zero but whose linear part is not identically zero. Llibre et al. \cite{Dumortier2007} proved that at isolated singularities an analytic system has a finite senatorial decomposition, which means only finite steps are needed to blow-up a system sufficiently. 

Consider the planar homogeneous polynomial vector field of degree $n>1$
\begin{equation}\label{eqn-5}
	\mathcal{H}_{n}:\dot{x}=\sum_{i+j=n}c_{ij}x^{i}y^{j}:=P_{n}(x,y),\quad\dot{y}=\sum_{i+j=n}d_{ij}x^{i}y^{j}:=Q_{n}(x,y),
\end{equation}
with $c_{ij},d_{ij}\in\mathbb{R}$ not all equal to zero, and $P_n$ and $Q_n$ coprime. We take the change of variables
\begin{equation}\label{eqn-6}
	x=x, \quad y=ux,
\end{equation}
and transform system \eqref{eqn-5} into 
\begin{equation}\label{eqn-7}
	\begin{aligned}
		\frac{dx}{d\tau} &= xP_n(1,u) = x\sum_{i+j=n} c_{ij} u^j, \\
		\frac{du}{d\tau} &= G(1,u) = \sum_{i+j=n} (d_{ij} u^j - c_{ij} u^{j+1}),
	\end{aligned}
\end{equation}
after the time re-scaling $dt=d\tau/x^{n-1}$, where 
\begin{equation}\label{eqn-8}
	G(x,y):=xQ_n(x,y)-yP_n(x,y).
\end{equation}
Then the blow-up transformation decomposes the origin of system \eqref{eqn-5} into several simpler ones located on the $u$-axis of system \eqref{eqn-7}. Thus the zeros of $G(1,u)$ determine the singularities located on the $u$-axis of system \eqref{eqn-7}, which correspond to the characteristic directions of system \eqref{eqn-5} at the origin.

 Cima and Llibre \cite{CIMA1990420} presented a classification of the singularity at the origin of system \eqref{eqn-5}. Tang and Zhang \cite{TANG201990} provided more detail conditions for characterizing the singularity.

Suppose that there exists a $u_0$ satisfying $G(1,u_0)=0$, then $E=(0, u_0)$ is a singularity of system \eqref{eqn-7} and system \eqref{eqn-5} has the characteristic direction $\theta=\arctan(u_0)$ at the origin. Besides, $P_n(1,u_0)\neq 0$ since $G(1,u_0)=0$. Otherwise, $y-u_0x$ is a common factor of $Q_n(x,y)$ and $P_n(x,y)$, which is a contradiction with the fact that $Q_n(x,y)$ and $P_n(x,y)$ have not a non-constant common factor. Denote the derivative of $G(1,u)$ with respect to $u$ at $u_0$ by $G^{\prime}(1,u_0)$. Then we have the following lemma from  \cite{TANG201990}.

\begin{lemma}
	For the singularity $E=(0, u_0)$ of system \eqref{eqn-7}, the following statements hold:

(a)  For $G'(1, u_0) \neq 0$, the singularity $E$ is
	
	- either a saddle if $P_{n}( 1, u_{0}) G^{\prime }( 1, u_{0}) < 0$,
	
	- or a node if $P_{n}( 1, u_{0}) G^{\prime }( 1, u_{0}) > 0.$
	
	(b) For $G'(1, u_0)=0$, if $u_0$ is a zero of multiplicity $m> 1$ of $G( 1, u)$ , then $E$ is	

	- either a saddle if $m$ is odd and $P_{\boldsymbol{n}}(1,u_{0})G^{(m)}(1,u_{0})<0$,

	- or a node if $m$ is odd and $P_{n}(1,u_{0})G^{(m)}(1,u_{0})>0$,

	- or a saddle- node if $m$ is even. 
\end{lemma} 

Note that the change  \eqref{eqn-6} is singular on $x=0$. We now determine the number of orbits approach the origin and are tangent to the $y$–axis. We here apply the method of normal sector. For system \eqref{eqn-5}, the origin is the unique finite singularity. 

	Let $\triangle OAB$ be a sector-shaped region bounded by the radius vectors $\overrightarrow{OA}$ and $\overrightarrow{OB}$, centered at the singular point $O$, and enclosed by the arc $AB$. The region $\triangle OAB$ is called a \textit{normal sector} if the following conditions are satisfied:
	\begin{enumerate}
		\item The region $\triangle OAB$ contains no singular points other than $O$, and the segments $OA$ and $OB$ are regular (i.e., non-singular) except at the point $O$;
		\item At any point within $\triangle OAB$, the direction field is not perpendicular to the coordinate axes;
		\item The region $\triangle OAB$ contains at most             one characteristic direction, and neither              $\overrightarrow{OA}$ nor                   $\overrightarrow{OB}$ lies along a characteristic   direction (i.e., the angles they form with the $x$-axis             are not characteristic).
	\end{enumerate}

Obviously there only three classes of normal sectors, showed separately in Figure 1. Then we have the following criteria:
\begin{itemize}
	\item[(i)] If $\triangle AOB$ is a normal sector in the first class of Figure 1, then system \eqref{eqn-5} has infinitely many orbits approaching $O$ in $\triangle AOB$ as $t \to +\infty$.
	
	\item[(ii)] If $\triangle AOB$ is a normal sector in the second class of Figure 1, then system \eqref{eqn-5} has either a unique orbit approaching $O$ in $\triangle AOB$ as $t \to +\infty$ or infinitely many orbits approaching $O$ in $\triangle AOB$ as $t \to +\infty$.
	
	\item[(iii)] If $\triangle AOB$ is a normal sector in the third class of Figure 1, then system \eqref{eqn-5} has either no orbits approaching $O$ in $\triangle AOB$ as $t \to +\infty$ or infinitely many orbits approaching $O$ in $\triangle AOB$ as $t \to +\infty$.
\end{itemize}

\begin{figure}[h]
	\centering
	\includegraphics[width=0.8\textwidth]{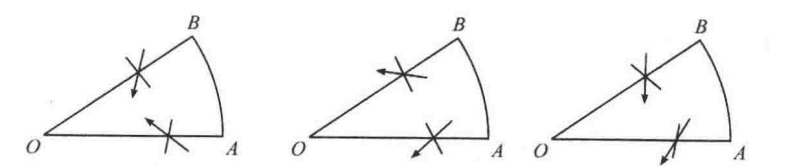} 
	\caption{Three classes of normal sectors}
	\label{fig:example}
\end{figure}

Applying the polar coordinate change $x=r\cos\theta$ and $y=r\sin\theta$, system \eqref{eqn-5} can be written in 
\begin{equation}
	\frac{1}{r}\frac{dr}{d\theta}=\frac{\tilde{H}(\theta)}{\tilde{G}(\theta)},
\end{equation}
where
$$\begin{aligned}\tilde{G}(\theta)&=G(\cos\theta,\sin\theta),\\\tilde{H}(\theta)&=H(\cos\theta,\sin\theta),\\H(x,y)&:=yQ_{n}(x,y)+xP_{n}(x,y).\end{aligned}$$

Hence a necessary condition for $\theta=\theta_0$ to be a characteristic direction at the origin is $G(\cos\theta_{0},\sin\theta_{0})=0$.

If $u_0$ is a root of equation $G(1,u)=0$, then $H(1,u_0)\neq 0$ since 
\begin{equation}
	H(1,u_0)=(1+u_0^2)P_n(1,u_0),
\end{equation}
and $P_n(1,u_0)\neq 0$ from the fact that $Q_n(x,y)$ and $P_n(x,y)$ are coprime.

The next lemma is about the trajectories of system \eqref{eqn-5} at the origin along the characteristic direction $\theta=\frac{\pi}{2}$, i.e. $G(v,1)=0$ at $v=0$, where $v=\frac{x}{y}$.

\begin{lemma}
	Assume that $\theta=\frac{\pi}{2}$ is a zero of multiplicity $m$ of $\tilde{G}(\theta)$. The following statements hold:
	
	(a) If $m > 0$ is even, there exists infinitely many orbits connecting the origin of system \eqref{eqn-5} and being tangent to the $y$–axis at the origin.
	
	(b) If $m>0$ is odd, there exists either infinitely many orbits if $\widetilde{G}^{(m)}(\frac{\pi}{2})\widetilde{H}(\frac{\pi}{2})>0$, or exactly one orbit if $\widetilde{G}^{(m)}(\frac{\pi}{2})\widetilde{H}(\frac{\pi}{2})<0$, connecting the origin of system \eqref{eqn-5} and being tangent to the $y$–axis at the origin.
\end{lemma}

In order to get the concrete parameter conditions determining the dynamics of system \eqref{eqn-1} at infinity, we need the \text{Poincar\'{e}}
 compactification. Generally, the discussion can be divided into some steps showed in Zhang, Ding's book \cite{zhang1985}.

First the \text{Poincar\'{e}}
 transformation 
\begin{equation}\label{eqn-9}
	x=\frac{1}{z},\quad y=\frac{u}{z},
\end{equation}
transforms system \eqref{eqn-1} into 
\begin{equation}\label{eqn-10}
	\begin{aligned}
		\frac{du}{dt} &= -uzP\left(\frac{1}{z},\frac{u}{z}\right) + zQ\left(\frac{1}{z},\frac{u}{z}\right), \\
		\frac{dz}{dt} &= -z^2P\left(\frac{1}{z},\frac{u}{z}\right).
	\end{aligned}
\end{equation}

After the time re-scaling $d\tau=dt/z^n$, the system \eqref{eqn-10} can be written as
\begin{equation}\label{eqn-11}
	\frac{du}{d\tau}=P^{*}\left(u,z\right),\quad\frac{dz}{d\tau}=Q^{*}\left(u,z\right),
\end{equation}
where the value of $n$ 	makes $P^{*}$, $Q^{*} $ irreducible polynomial.

Solve the equation $P^*\left(u,0\right)=0$, $Q^*\left(u,0\right)=0$ and classify each singularity $B(u,0)$, which shows the structure at the end of $y=ux$ of the system \eqref{eqn-1}.

Then we take \text{Poincar\'{e}}
 transformation
\begin{equation}\label{eqn-12}
	x=\frac{v}{z},\quad y=\frac{1}{z},
\end{equation}
to transform system \eqref{eqn-1} into
\begin{equation}\label{eqn-13}
	\begin{aligned}
		\frac{dv}{dt} &= zP\left(\frac{v}{z},\frac{1}{z}\right) - zvQ\left(\frac{v}{z},\frac{1}{z}\right) = \frac{\hat{P}(v,z)}{z^m}, \\
		\frac{dz}{dt} &= -z^2Q\left(\frac{v}{z},\frac{1}{z}\right) = \frac{\hat{Q}(v,z)}{z^m},
	\end{aligned}
\end{equation}
with similar time re-scaling $d\tau=dt/z^{m}$.

Classifying the origin of the system
\begin{equation}\label{eqn-14}
	\frac{dv}{d\tau}=\hat{P}(v,z),\quad\frac{dz}{d\tau}=\hat{Q}(v,z),
\end{equation}
will show the structure at the end of the $y$-axis of the system \eqref{eqn-1}.

We have the following lemma from García et al. \cite{GARCIA20133185} to have some basic analysis of the systems.
\begin{lemma}
	
	(a) The vanishing set of each linear factor (if exists) of $G(x, y)$ is an invariant line of the homogeneous system.
	
	(b) The origin of system \eqref{eqn-5} is a global center if and only if $$\int_{-\infty}^{\infty}\frac{P_{n}(1,u)}{G_{n}(1,u)}du=0.$$
\end{lemma}

From the expression of $G(x,y)$ in \eqref{eqn-8}, we know that the origin of a homogeneous system cannot be a center if the degree of the system is even. Actually, in this case there exists a $u_0$ satisfying $G(1,u_0)=0$, and the system has a invariant line $y=u_0x$.

\bigskip


\section{Quasi-homogeneous polynomial systems of degree 2 and 3}
\label{s3}


In 2019, García et al. \cite{GARCIA20133185} provided an algorithm to compute quasi-homogeneous but non-homogeneous polynomial differential systems with a given degree. From the definition of the quasi-homogeneous polynomial system, they gave the necessary condition $P_n, Q_n$ should satisfy in system \eqref{eqn-3}, \eqref{eqn-4}. Thus, for a given degree of the system, their algorithm identify the monomials in the system in order according to the necessary condition and get all the possible forms. In this paper they also obtain all the quadratic and cubic quasi-homogeneous but non-homogeneous vector fields.
\begin{lemma}\label{lemma-1}
	A  quadratic quasi-homogeneous but non-homogeneous polynomial differential system after a re-scaling of the variables can be written in one of the following forms:
	
	(2a) $x^{\prime}=y^{2},y^{\prime}=x$, with minimal weight vector $(3,2,2)$.
	
	(2b) $x^{\prime}=axy,y^{\prime}=x+y^{2}$, with $a\neq0$ and minimal weight vector$(2,1,2)$.
	
	(2c) $x^{\prime}=x+y^{2},y^{\prime}=ay$, with $a\neq0$ and minimal weight vector $(2,1,1)$.
\end{lemma}
\begin{lemma}\label{lemma-2}
	A  cubic quasi-homogeneous but non-homogeneous differential system after a re-scaling of the variables can be written in one of the following forms:
	
	(3a) $x^{\prime}=y(ax+by^{2}),y^{\prime}=x+y^{2}$, with $a\neq b$, or $x^{\prime}=y(ax\pm y^{2}),y^{\prime}=x$, and both with minimal weight vector $(2,1,2)$.
	
	(3b) $x^{\prime}=x^{2}+y^{3},y^{\prime}=axy$, with $a\neq0$ and minimal weight vector $(3,2,4)$.
	
	(3c) $x^{\prime}=y^{3},y^{\prime}=x^{2}$, with minimal weight vector $(4,3,6)$.
	
	(3d) $x^{\prime}=x(x+ay^{2}),y^{\prime}=y(bx+y^{2})$, with $(a,b)\neq(1,1)$ and minimal weight vector $(2,1,3)$.
	
	(3e) $x^{\prime}=axy^2,y^{\prime}=\pm x^2+y^3$, with $a\neq0$ and minimal weight vector $(3,2,5)$.
	
	(3f) $x^{\prime}=axy^{2},y^{\prime}=x+y^{3}$, with $a\neq0$ and minimal weight vector $(3,1,3)$.
	
	(3g) $x^{\prime}=ax+y^{3},y^{\prime}=y$, with $a\neq0$ and minimal weight vector is $(3,1,1)$.
\end{lemma}

\bigskip


\section{The corresponding homogeneous polynomial systems}

Tang and Zhang \cite{TANG201990} provided a method to study global dynamics of planar quasi-homogeneous differential systems. They first proved that all planar quasi-homogeneous polynomial differential systems can be transformed into homogeneous differential systems.
\begin{lemma}
	Any quasi-homogeneous but non-homogeneous polynomial differential system \eqref{eqn-1} of degree $n$
	can be transformed into a homogeneous polynomial differential system by a change of variables being of the
	composition of the transformation $\tilde{x}=(\pm x)^{\frac{s_2}{\beta}}, \tilde{y}=(\pm y)^{\frac{s_1}{\beta}}$, where $\beta$ is a suitable non-negative integer.
\end{lemma}

\noindent\textbf{Remark 1}\quad
	From \cite{TANG201990} we know that the least common multiple of $s_1$ and $s_2$ is not the unique choice for $\beta$. Actually, we can select any common multiple of $s_1$ and $s_2$ for $\beta$. For some systems we can use the change of variables 
	\begin{equation}\label{eqn-15}
		\tilde{x}=x^{s_{2}},\quad\tilde{y}=y^{s_{1}}.
	\end{equation}	 
This method is more effective than analyzing the quasi-homogeneous polynomial system directly in many conditions. First, we have more conclusions for homogeneous polynomial vector fields, which helps us get the structure of the original system clearly. Second, transforming quasi-homogeneous polynomial systems to homogeneous ones can simplify the systems both in degree and in types in some conditions. For example, \cite{TANG201990} proved the 15 systems of planar quintic quasi-homogeneous but non-homogeneous polynomial differential systems can be transformed to four homogeneous systems with parameters, and their degree no more than three. Thus we only need to research the parameter conditions for simpler systems. Moreover, after the transformation, from the former discussions we can get that the separatrix of the homogeneous polynomial system is a line through the origin. This helps us to make the structure of the system clear in each sector.

Another useful lemma for us to get the global phase portraits is about symmetric. In \cite{TANG201990} it is proved that quasi-homogeneous systems are symmetric with respect to either an axis or the origin with possibly a time reverse. Using this fact we can obtain global dynamics of a quasi-homogeneous differential system through the dynamics of its associated homogeneous systems in a half plane. 

\begin{lemma}
	An arbitrary quasi-homogeneous but non-homogeneous polynomial differential system of degree $n$ with the minimal weight vector $(s_1, s_2, d)$ is invariant either under the change $(x,y)\to(x,-y)$ if $s_1$ is even and $s_2$ is odd, or under the change $(x,y)\to(-x,y)$ if $s_1$ is odd and $s_2$ is even, or under the change $(x,y)\to(-x,-y)$ if both $s_1$ and $s_2$ are odd, without taking into account the direction of the time.
\end{lemma}
\noindent\textbf{Remark 2}\quad
	From \cite{TANG201990} corollary 3 we know that if a quasi-homogeneous but non-homogeneous differential system \eqref{eqn-1} of degree $n$ has the minimal weight vector $\tilde{w}=(s_{1},s_{2},d)$, then at least one of $s_1$ and $s_2$ is odd.

Using these lemmas we can transform the quasi-homogeneous polynomial systems of degree two and three to homogeneous forms.
\begin{theorem}
	 A quasi-homogeneous but non-homogeneous quadratic polynomial differential system when restricted to either $y>0$ or $x>0$ can be transformed to one of the following two homogeneous systems:
	  
	  $\mathcal{H}_{1}$: $\dot{x}=c_{01}y+c_{10}x, \dot{y}=d_{01}y+d_{10}x$, with $c_{01}d_{10}\neq 0$, or $d_{10}=0, c_{01}c_{10}d_{01}\neq0$, or $c_{01}=0,c_{10}d_{01}d_{10}\neq0$,
	  
	  $\mathcal{H}_{0}$: $\dot{x}=c_{0},\dot{y}=d_{0}$, with $c_{0}d_{0}\neq 0$.
\end{theorem}
	
\begin{proof}
	Using the change \eqref{eqn-15} and some time re-scaling, some calculations show that the quadratic quasi-homogeneous system in Lemma \ref{lemma-1} can be correspondingly transformed into a homogeneous one:
	
	$\tilde{(2a)}$: $\dot{x}= 2, \dot{y}=3 $, by $dt=dt_1/x^{\frac{1}{2}}y^{\frac{2}{3}}$, $x>0$,
	
	$\tilde{(2b)}$: $\dot{x}=ax, \dot{y}=2(x+y)$, with $a\neq 0$, by $dt=dt_1/y^{\frac{1}{2}}$, $y>0$,
	
	$\tilde{(2c)}$: $\dot{x}=x+y, \dot{y}=2ay$, with $a\neq 0$, $y>0$,
	where we still use $x, y$ replacing $\tilde{x}, \tilde{y}$ for simplifying notations. Note that each system is one of the two systems in the theorem. We complete the proof of the theorem.
\end{proof}
	
\begin{theorem}
	A quasi-homogeneous but non-homogeneous cubic polynomial differential system when restricted to either $y>0$ or $x>0$ can be transformed to one of the following three homogeneous systems:
	
	$\mathcal{H}_{2}:\dot{x}=c_{02}y^{2}+c_{11}xy+c_{20}x^{2}, \dot{y}=d_{02}y^{2}+d_{11}xy+d_{20}x^{2}$, with $c_{02}d_{20}\neq0$, or $c_{02}=d_{20}=0$, $c_{20}d_{02}\neq0$,

	$\mathcal{H}_{1}$: $\dot{x}=c_{01}y+c_{10}x, \dot{y}=d_{01}y+d_{10}x$, with $c_{01}d_{10}\neq 0$, or $d_{10}=0, c_{01}c_{10}d_{01}\neq0$, or $c_{01}=0,c_{10}d_{01}d_{10}\neq0$,
	
	$\mathcal{H}_{0}$: $\dot{x}=c_{0},\dot{y}=d_{0}$, with $c_{0}d_{0}\neq 0$.
\end{theorem}
	
\begin{proof}
	Using the change \eqref{eqn-15} and some time re-scaling, some calculations show that the cubic quasi-homogeneous system in Lemma \ref{lemma-2} can be correspondingly transformed into a homogeneous one:
	
	$\tilde{(3a)}$: $\dot{x}= ax+by, \dot{y}=2(x+y) $, with $a\neq b,$ or $\dot{x}= ax\pm y, \dot{y}=2x $, by $dt=dt_1/y^{\frac{1}{2}}$, $y>0$,
	
	$\tilde{(3b)}$: $\dot{x}=2(x+y), \dot{y}=3ay$, with $a\neq 0$, by $dt=dt_1/x^{\frac{1}{2}}$, $x>0$,
	
	$\tilde{(3c)}$: $\dot{x}=3, \dot{y}=4$, by $dt=dt_1/x^{\frac{2}{3}}y^{\frac{3}{4}}$, $y>0$,
	
	$\tilde{(3d)}$: $\dot{x}=x(ay+x), \dot{y}=2y(bx+y) $, $(a,b)\neq (1,1)$, $y>0$,
	
	$\tilde{(3e)}$: $\dot{x}=2ax, \dot{y}=3(y\pm x)$, with $a\neq 0$, by $dt=dt_1/y^{\frac{2}{3}}$, $x>0$,
	
	$\tilde{(3f)}$: $\dot{x}=ax, \dot{y}=3(x+y)$, with $a\neq 0$, by $dt=dt_1/y^{\frac{2}{3}}$, $x>0$,
	
	$\tilde{(3g)}$: $\dot{x}=ax+y, \dot{y}=3y$, with $a\neq 0$, by $dt=dt_1/y^{\frac{2}{3}}$, $x>0$,
	where we still use $x, y$ replacing $\tilde{x}, \tilde{y}$ for simplifying notations. Note that each system is one of the three systems in the theorem. We complete the proof of the theorem.
\end{proof}

\noindent\textbf{Remark 3}\quad
	For the quasi-homogeneous systems which are symmetric with respect to the origin, we only research them on $x>0$ without loss of generality.

After transforming the quasi-homogeneous polynomial systems to homogeneous ones, we finally only need to characterize system $\mathcal{H}_2, \mathcal{H}_1, \mathcal{H}_0$ to get the global structure of the original systems.

\section{Global dynamics and phase portrait analysis}
First we study topological structures of the homogeneous system $\mathcal{H}_{2}$, which corresponds to the quasi-homogeneous systems $(3d)$ and homogeneous system $\tilde{(3d)}$, which is the most complicated system.
\begin{proposition}
	The global phase portraits of the homogeneous system $\tilde{(3d)}$ is topologically equivalent to one of the five phase portraits showed in Figure 2 and Figure 3 without taking into account the direction of the time.
\end{proposition}

\begin{proof}
	
	We set \begin{equation}\label{eqn-16}
	\begin{aligned}
		\dot{x} &= x^2 + a_{12}xy := P_2(x,y), \\
		\dot{y} &= 2b_{12}xy + 2y^2 := Q_2(x,y),
	\end{aligned}
\end{equation}
	and $G_{2}(x,y)=xQ_{2}(x,y)-yP_{2}(x,y).$

	Taking  respectively the \text{Poincar\'{e}}
	 transformations $x=1/z,y=u/z$  and  $x=v/z,y=1/z$  together with the time re-scaling $dt_{1}=zdt$, system \eqref{eqn-16} around the equator of the \text{Poincar\'{e}}
	  sphere can be written respectively in 
	\begin{equation}
	\begin{aligned}
		\dot{u} &= (2 - a_{12})u^2 + (2b_{12} - 1)u, \\
		\dot{z} &= -z(1 + a_{12}u),
	\end{aligned}
\end{equation}

	and
	\begin{equation}
	\begin{aligned}
		\dot{v} &= (a_{12} - 2)v + (1 - 2b_{12})v^{2}, \\
		\dot{z} &= -2z(1 + b_{12}v).
	\end{aligned}
\end{equation}

	Therefore, there exists a singularity $I_1$ located at the end of the $y$-axis. From Lemma 1 we know that $I_1$ is a saddle if $a_{12}>2$ or  a node if $a_{12}<2$. When $a_{12}=2$ and $b_{12}\neq \frac{1}{2}$, $I_1$ is a saddle-node. Notice that when $a_{12}=2$, $b_{12}=\frac{1}{2}$, $P_2(x,y)$ and $Q_2(x,y)$ have the common factor $x+2y$, which is not under our discussion. Thus we have done the classification of the singularity $I_1$. There also exists a singularity $I_0$ at the end of the $x$-axis. In our following discussion, we can see that the trajectory structure near the singularity $I_0$ is related to the category of the singularity $(0,0)$. 
	
	Consider the blow-up at the origin, if $u_0$ is a zero of $G_2(1,u)$, which means $E_1(0,u_0)$ is a singularity. Moreover, $u_0$ also a singularity for system (17), which means there's also a singularity $E_2$ located at the end of the $y=u_0x$ for system (16). The Jacobian matrix at the two points $E_1$ and $E_2$ are as follows:
	$$J(E_1)=\begin{pmatrix}P_2(1,u_0)&0\\0&G^{\prime}(1,u_0)\end{pmatrix},\quad J(E_2)=\begin{pmatrix}-P_2(1,u_0)&0\\0&G^{\prime}(1,u_0)\end{pmatrix}.$$
	
    Then from Lemma 1 we know that if $E_2$ is a saddle (node, saddle-node), then $E_1$ is a node (saddle, saddle-node). Actually from the form of the \text{Poincar\'{e}} transformation we can know that it's almost a blow-up at the infinity, with $x=1/z$ to have a closer vision of infinity. Thus it's natural to have such relationship when the system doesn't have other singularities.

	Having the above preparation, we now study the global phase portraits of system $\tilde{(3d)}$ according to the number of zeros of $\widehat{G}_{2}(u):=G_{2}(1,u)=u(2b_{12}-1+(2-a_{12})u)$, which can have either one or two real roots. The direction of the orbits along the separatrix can be determined from the equation. From the equation we also notice that $x=0$ is an invariant line of the system. Then we have the following discussions.

	(a) $a_{12}\neq 2, b_{12}\neq \frac{1}{2}$: Then $\widehat{G}_{2}(u)$ has two zeros with multiplicity 1. Thus from Lemma 1, $u_0=0$ means $(0,0)$ is a node when $2b_{12}-1>0$ or a saddle when $2b_{12}-1<0$, $u_1=\frac{1-2b_{12}}{2-a_{12}}$ means $(u_1,0)$ is a node when $\frac{a_{12}b_{12}-1}{a_{12}-2}>0$ or a saddle when $\frac{a_{12}b_{12}-1}{a_{12}-2}<0$. As $\frac{\pi}{2}$  is a zero of multiplicity 1 of $\tilde{G}(\theta)$, there exists infinitely many orbits connecting the origin of system $\tilde{(3d)}$ and being tangent to the $y$–axis at the origin when $a_{12}>2$ and exactly one orbit when $a_{12}<2$. And for the zero point $(u,0)$, the singularity located at the end of $y=ux$ can be determined too from our previous discussion. In summary, the topological phase portraits of system $\tilde{(3d)}$ is determined by the sign of $A:=2b_{12}-1$, $B:=2(a_{12}-2)$ and $C:=2(1-a_{12}b_{12})$. We obtain the 3 possible topological global phase portraits shown in Figure 2. When two of $A,B,C$ are negative, one is positive, we have $(A)$. When two of $A,B,C$ are positive, one is negative, we have $(B)$. When $A,B,C$ are all negative, we have $(C)$. Obviously $A,B,C$ cannot be all positive. Thus we have finished the classification.  
	
	(b) $a_{12}\neq 2, b_{12}= \frac{1}{2}$: Then $u=0$ is a zero of multiplicity 2 so that $(0,0)$ is a saddle-node. As $\frac{\pi}{2}$  is a zero of multiplicity 1 of $\tilde{G}(\theta)$, there exists infinitely many orbits connecting the origin of system $\tilde{(3d)}$ and being tangent to the $y$–axis at the origin when $a_{12}>2$ or exactly one orbit when $a_{12}<2$. Considering the relationship between the category of the singularity (0,0) and $I_0$, we can know that $I_0$ is also a saddle-node. We obtain the global phase portraits showed in Figure 3.  Figure 3 $(A)$ is the case $a_{12}>2$, and  Figure 3 $(B)$ is the case $a_{12}<2$
	
	(c) $a_{12}=2, b_{12}\neq \frac{1}{2}$: When $b_{12}> \frac{1}{2}$, We can determine that $I_0$ is a saddle and $I_1$ is a saddle-node. When $b_{12}<\frac{1}{2}$, we can determine that $I_0$ is a saddle-node and $I_1$ is a saddle. Similarly, the phase portraits near the origin can be obtained too. The global phase portraits is the same as the two in Figure 3 topologically.

	\begin{figure}
		\centering
		\begin{subfigure}[b]{0.3\textwidth}
			\centering
			\includegraphics[width=\textwidth]{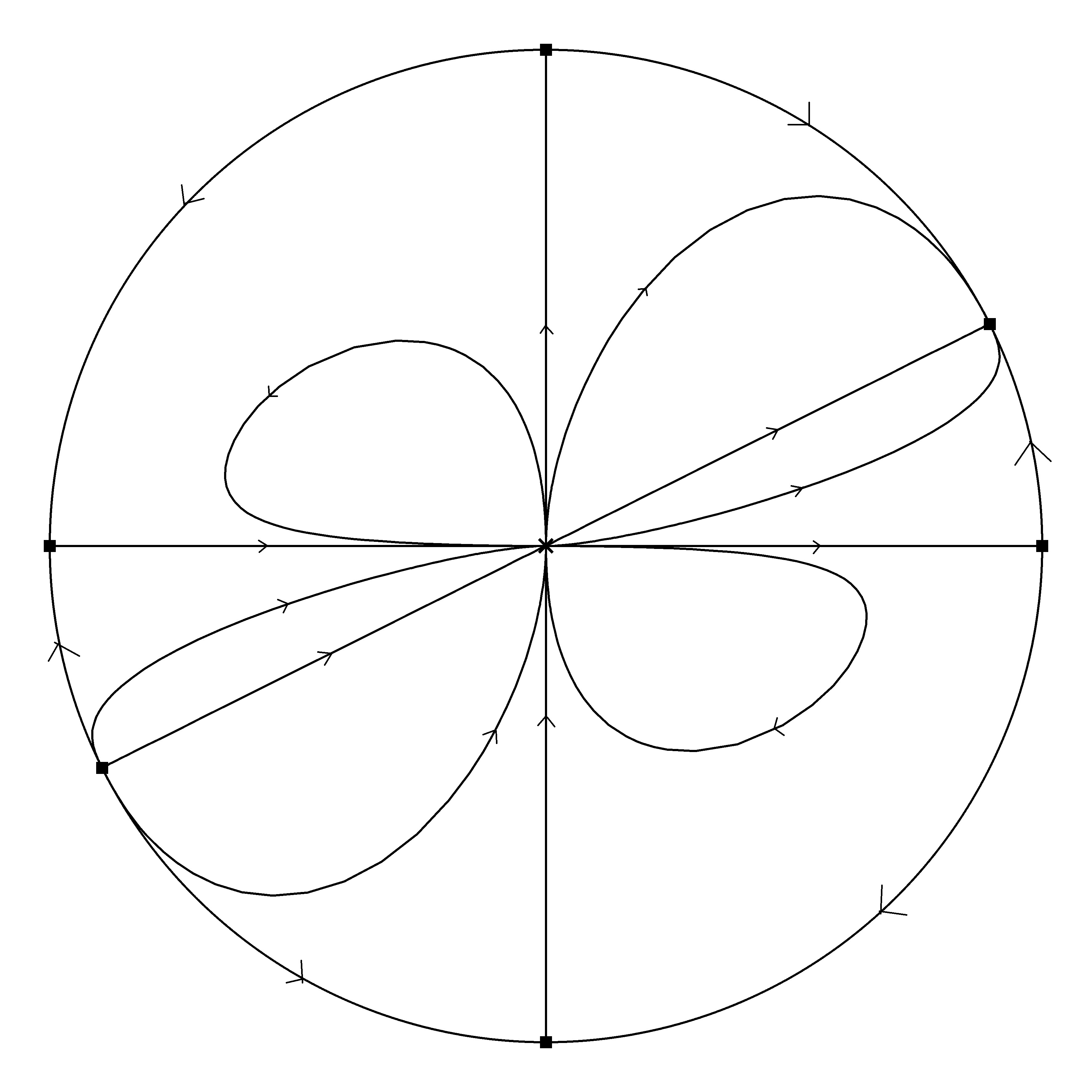}
			\caption{}
		\end{subfigure}
		\hfill
		\begin{subfigure}[b]{0.3\textwidth}
			\centering
			\includegraphics[width=\textwidth]{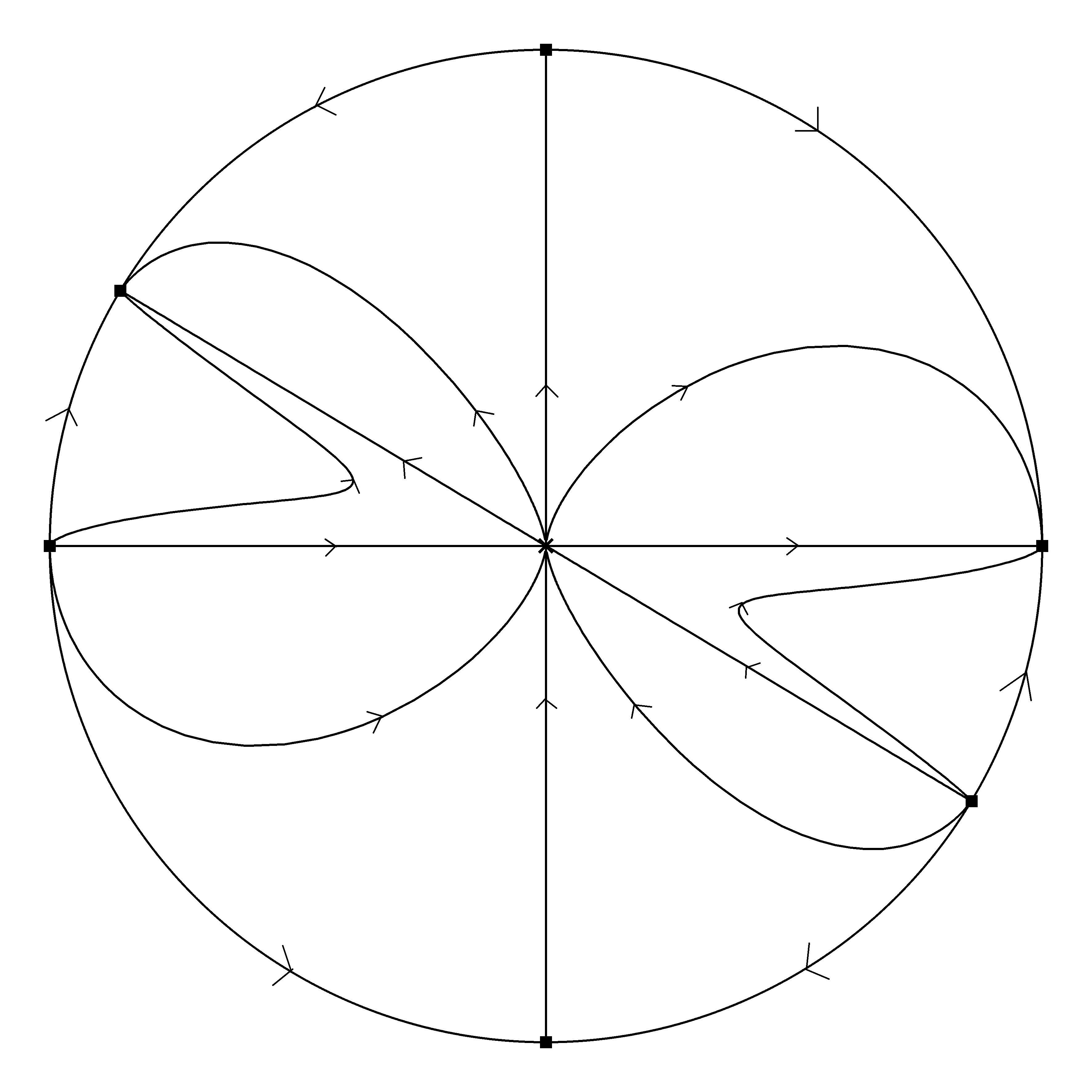}
			\caption{}
		\end{subfigure}
		\hfill
		\begin{subfigure}[b]{0.3\textwidth}
			\centering
			\includegraphics[width=\textwidth]{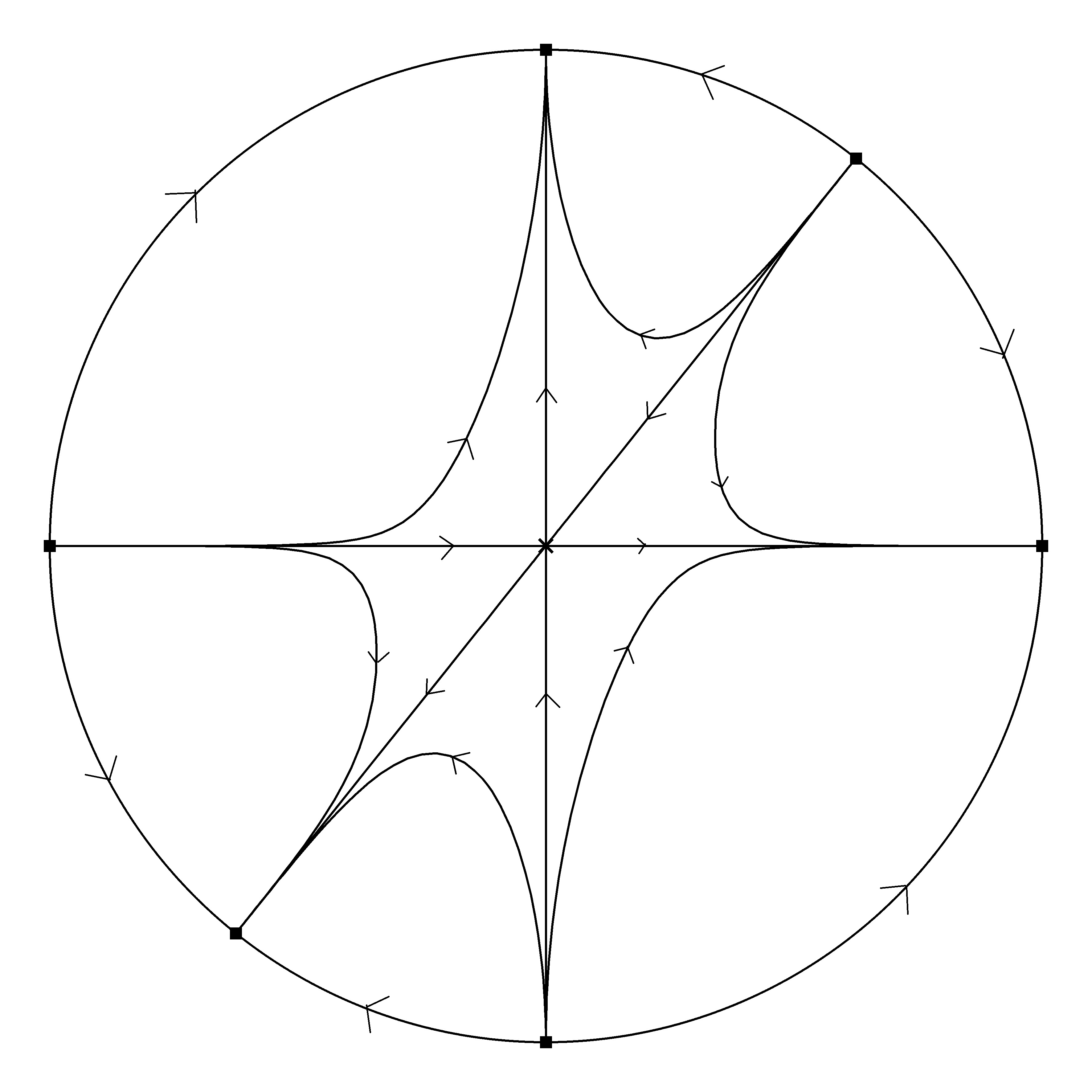}
			\caption{}
		\end{subfigure}
		\caption{Portraits of $\tilde{(3d)}$ when $\widehat{G}_{2}(u)$ has two different zeros}
		\label{fig:three_images}
	\end{figure}
	
	\begin{figure}[H]
	\centering
	\begin{subfigure}[b]{0.65\textwidth} 
		\centering
		\begin{subfigure}[b]{0.48\textwidth}
			\centering
			\includegraphics[width=\textwidth]{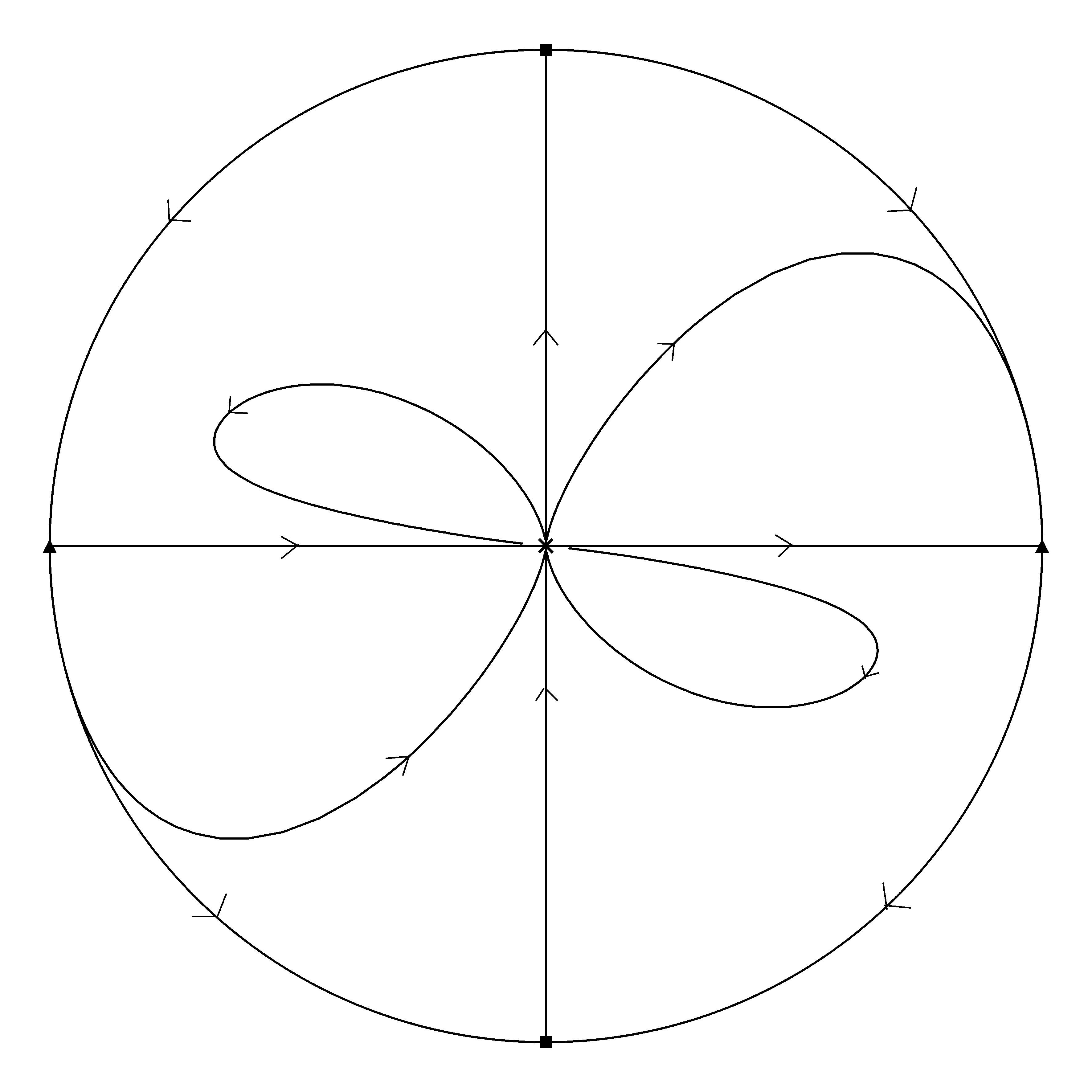}
			\caption{}
		\end{subfigure}
		\hfill
		\begin{subfigure}[b]{0.48\textwidth}
			\centering
			\includegraphics[width=\textwidth]{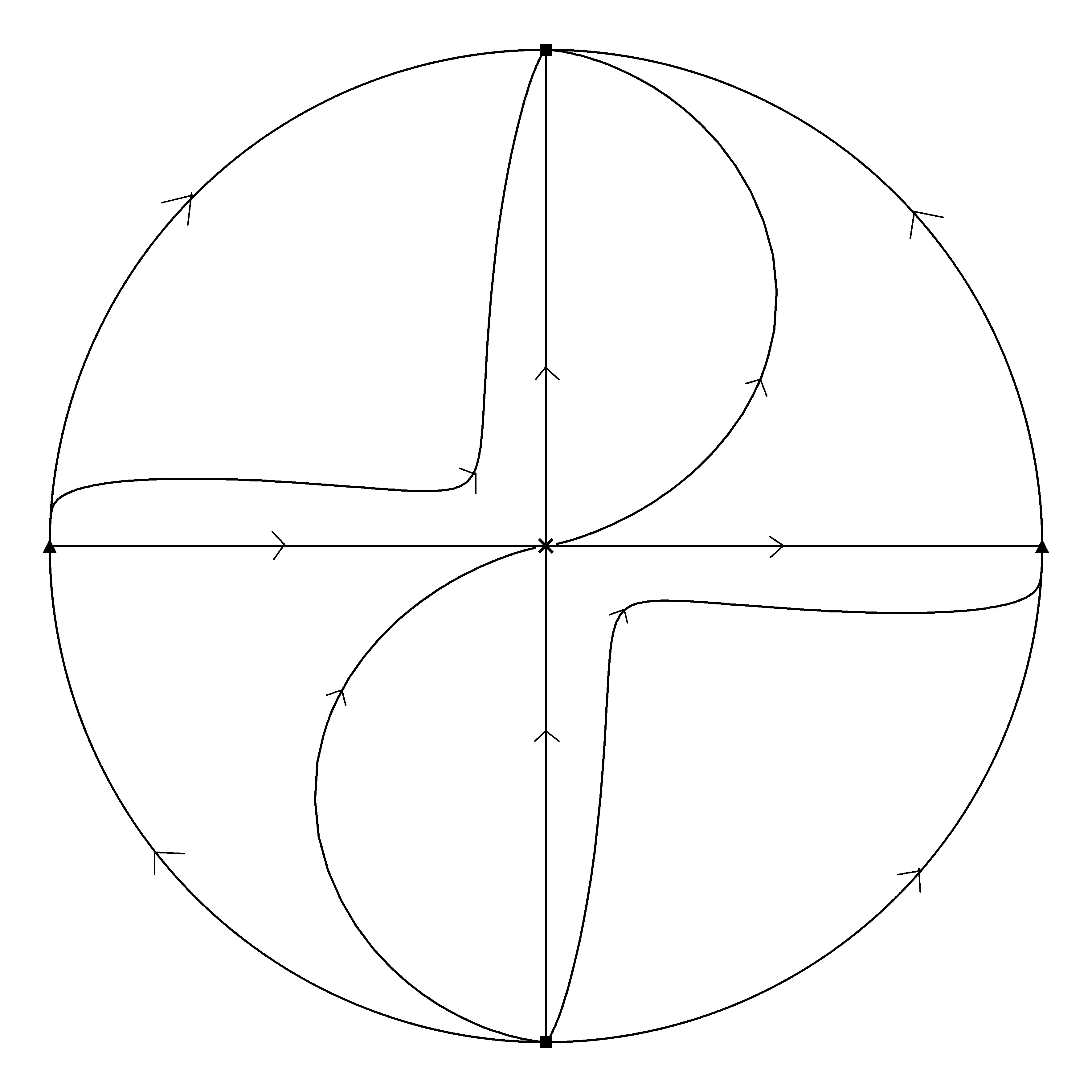}
			\caption{}
		\end{subfigure}
	\end{subfigure}
	
	\caption{Portraits of $\tilde{(3d)}$ when $\widehat{G}_{2}(u)$ has one zero of multiplicity 2}
	\label{fig:two_images}
\end{figure}

\end{proof}
\noindent\textbf{Remark 4}\quad 
	Here we don't analyze the origin directly just because when $u=0$ is a singularity, the structure after blow-down is not obvious to get. However, if we analyze the infinity directly, we can see the structure clearly in this example. Moreover, from the form of the system $\tilde{(3d)}$, we can see that the value of $a$ and $b$ shows symmetry by exchange $x$ and $y$. That's why the phase portraits are topologically the same. Moreover, it's easy to simplify that a blow-up in $y-$direction makes condition (c) almost the same as condition (b), only with the difference of parameters. Thus topologically they have the same global structure. 

	Now we consider the global structure of the original quasi-homogeneous system $(3d)$. 
\begin{theorem}
	The global phase portraits of the quasi-homogeneous system $(3d)$ is topologically equivalent to one of the eight phase portraits showed in Figure 4 and Figure 5 without taking into account the direction of the time.
\end{theorem}  
\begin{proof}
    Taking  respectively the \text{Poincar\'{e}}
	 transformations $x=1/z,y=u/z$  and  $x=v/z,y=1/z$  together with the time rescaling $dt_{1}=z^3dt$, system $(3d)$ around the equator of the \text{Poincar\'{e}}
	  sphere can be written respectively in 
	\begin{equation}
	\begin{aligned}
		\dot{u} &= (1 - a_{12})u^3 + (b_{12} - 1)u z, \\
		\dot{z} &= -z(z + a_{12} u^2),
	\end{aligned}
\end{equation}

	and
	\begin{equation}
	\begin{aligned}
		\dot{v} &= (a_{12} - 1)v + (1 - b_{12})v^{2} z, \\
		\dot{z} &= -z(1 + b_{12} v z).
	\end{aligned}
\end{equation}

    From the new systems  we know that when $a_{12}=1$, the infinity of system $(3d)$ is filled up with singularities. And system $(3d)$ has a singularity $I_0$ located at the end of the $x-$axis and $I_1$ located at the end of the $y-$axis when $a_{12}\neq 1$. 
	Through the change of variables between the quasi-homogeneous differential systems and their associated homogeneous ones we get that the invariant line $y=u_0x$ of system $\mathcal{H}_{2}$ as $u_0\neq 0$ corresponds to an invariant curve of system $\tilde{(3d)}$, which is tangent to the $y$-axis at the origin, and connects the origin and the singularity at infinity located at the end of $x$-axis.
	
	We also need to notice the following remarks. First, the transformation \eqref{eqn-15} is homeomorphism in  plane $y>0$, which doesn't change the global topological structure. Second, the separatrix will still connect the origin and the infinity singularity but not a straight line. Recall that it is also a reason why we want to transform the quasi-homogeneous system to homogeneous ones. Actually it can be calculated precisely through the variable changes, or we can use Taylor approximation of the stable and unstable manifolds. Third, the system is symmetric about $x=0$. Thus, we can get the following 5 kinds of portraits in Figure 4 when $a_{12}\neq 1$. Similar to this case, when $a_{12}=1$, there are 3 global portraits almost the same as (A), (B), (C) in Figure 5, which are different from those only now the infinity fulfills singularities instead of some invariant lines. We give them in Figure 4.

		\begin{figure}[H]
	\centering
	\begin{subfigure}[b]{0.3\textwidth}
		\centering
		\includegraphics[width=\textwidth]{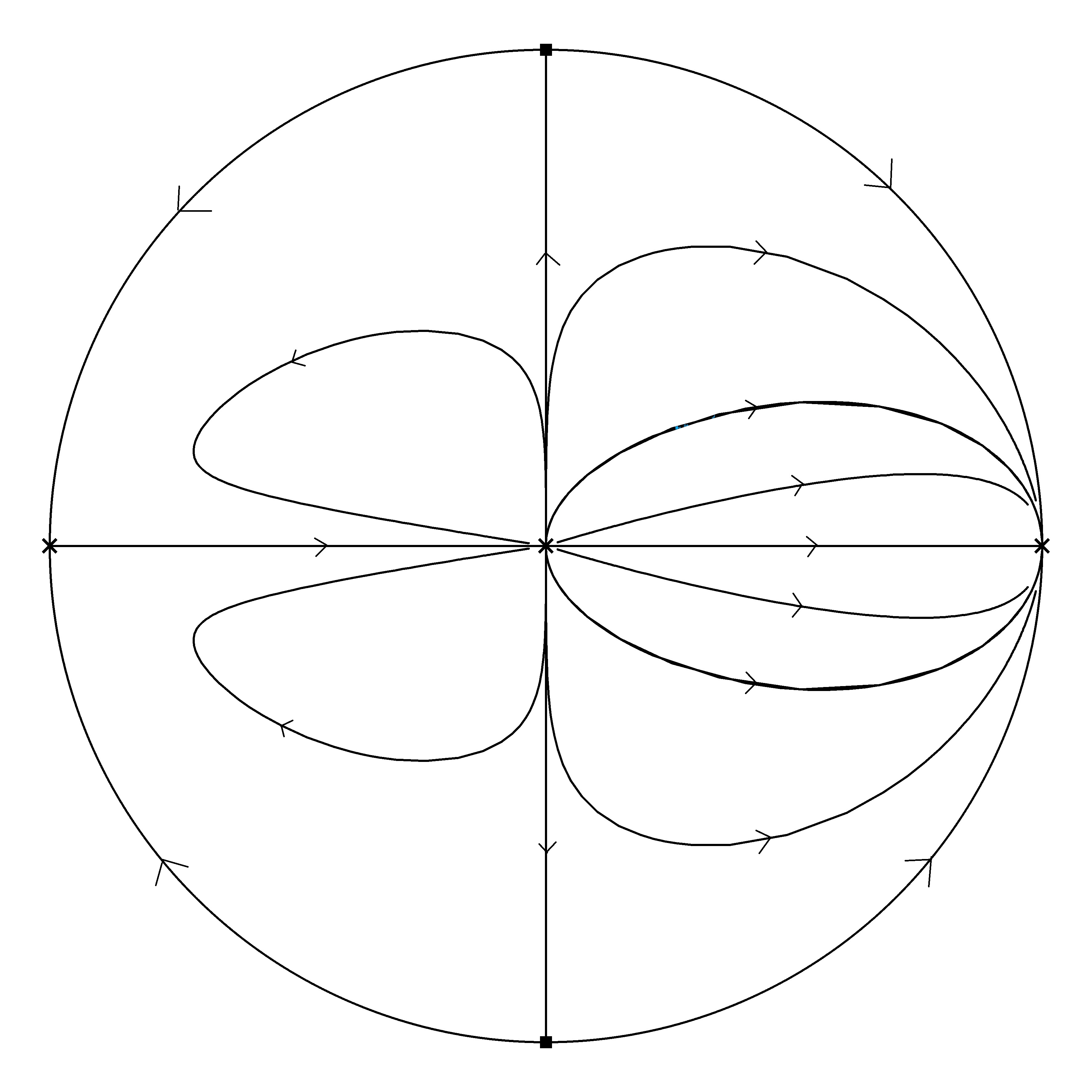}
		\caption{}
	\end{subfigure}
	\hfill
	\begin{subfigure}[b]{0.3\textwidth}
		\centering
		\includegraphics[width=\textwidth]{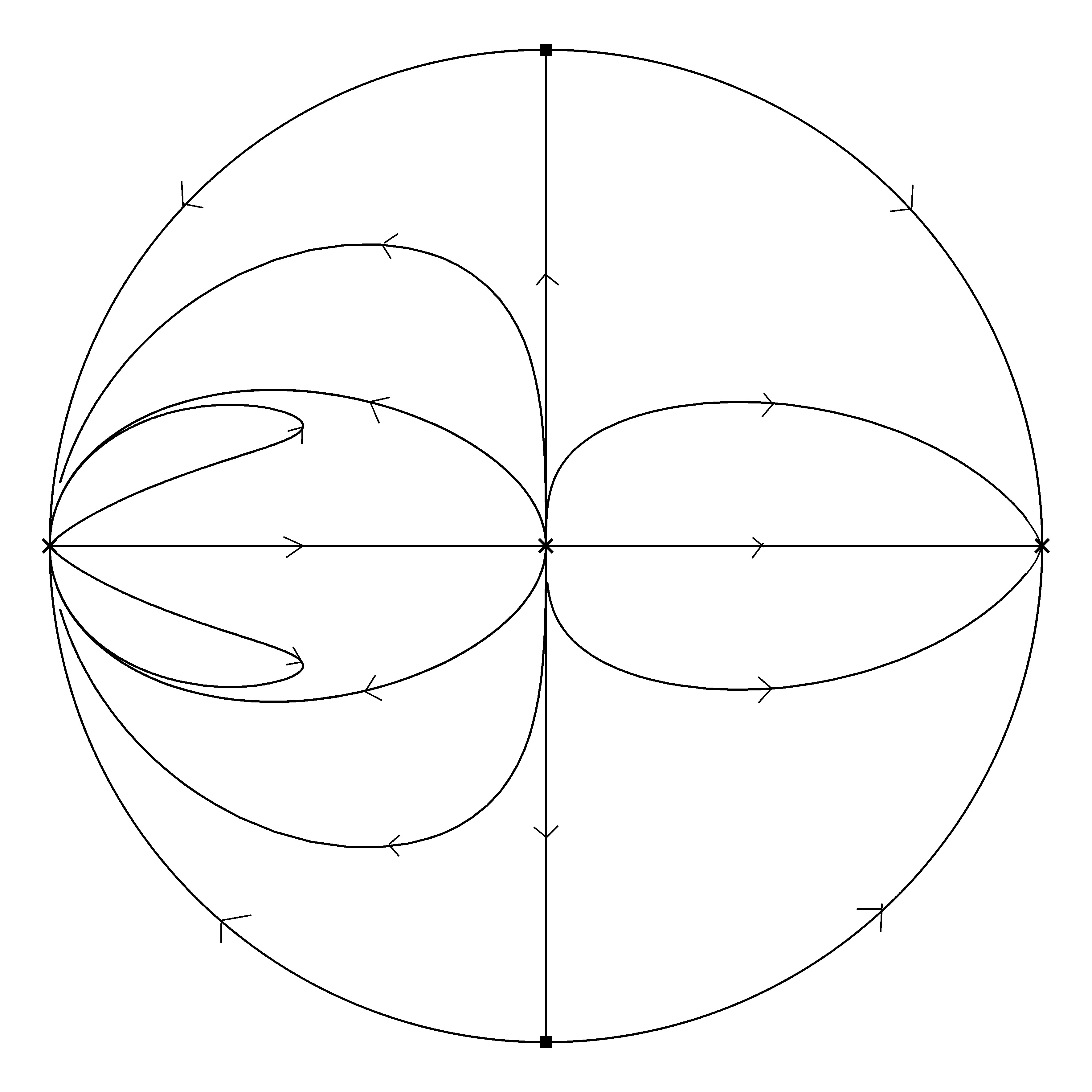}
		\caption{}
	\end{subfigure}
	\hfill
	\begin{subfigure}[b]{0.3\textwidth}
		\centering
		\includegraphics[width=\textwidth]{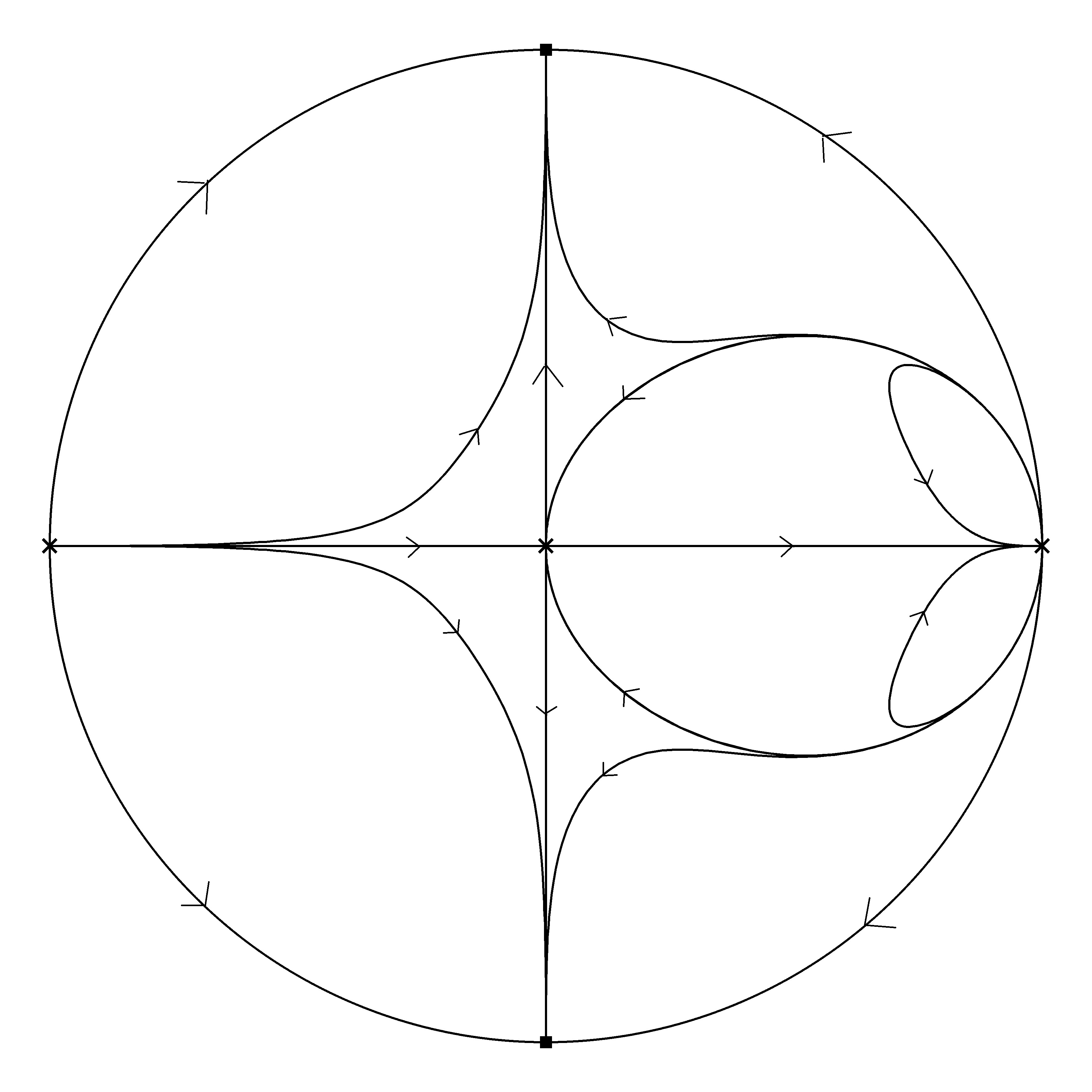}
		\caption{}
	\end{subfigure}

	\vspace{0.5cm} 

	\begin{subfigure}[b]{0.65\textwidth} 
		\centering
		\begin{subfigure}[b]{0.48\textwidth}
			\centering
			\includegraphics[width=\textwidth]{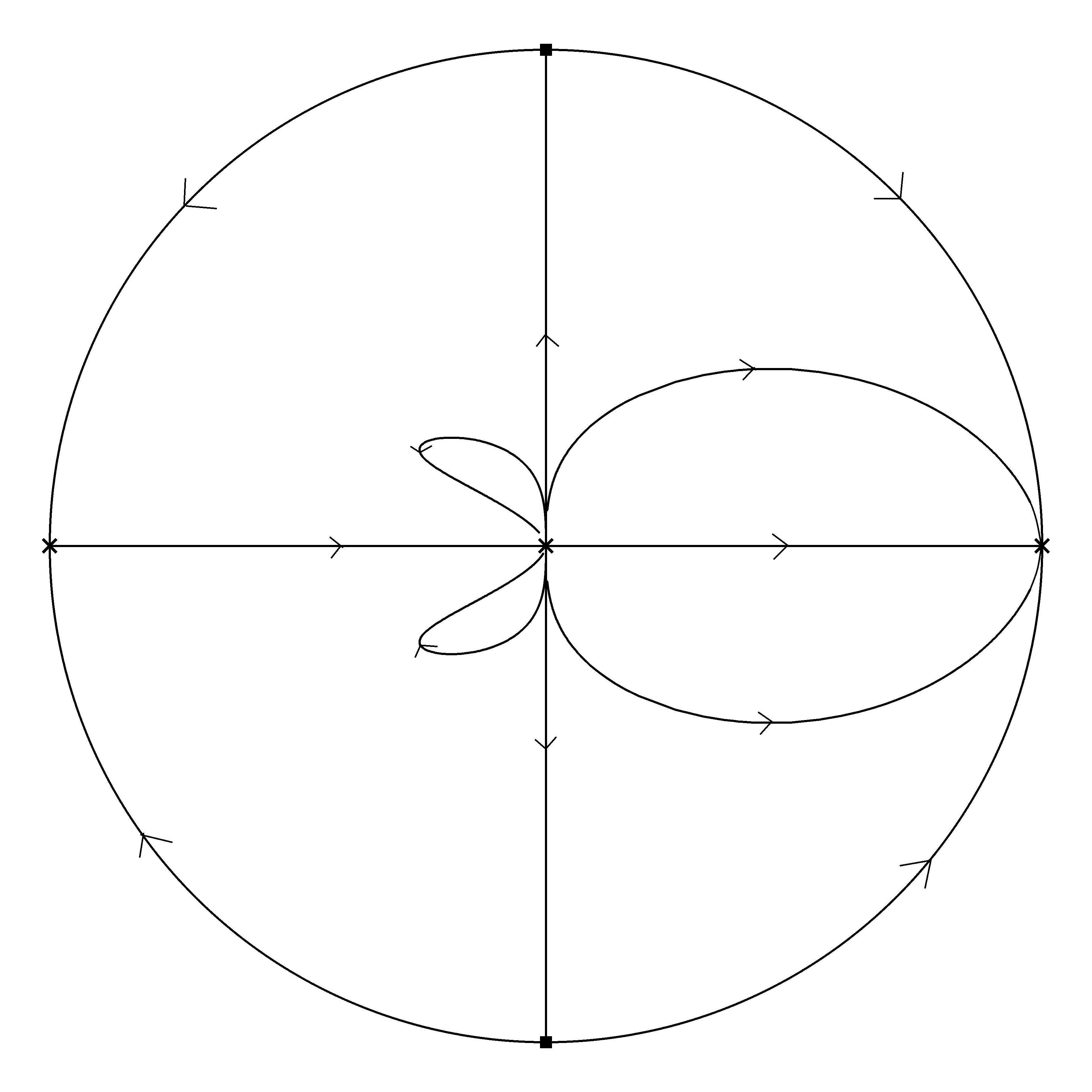}
			\caption{}
		\end{subfigure}
		\hfill
		\begin{subfigure}[b]{0.48\textwidth}
			\centering
			\includegraphics[width=\textwidth]{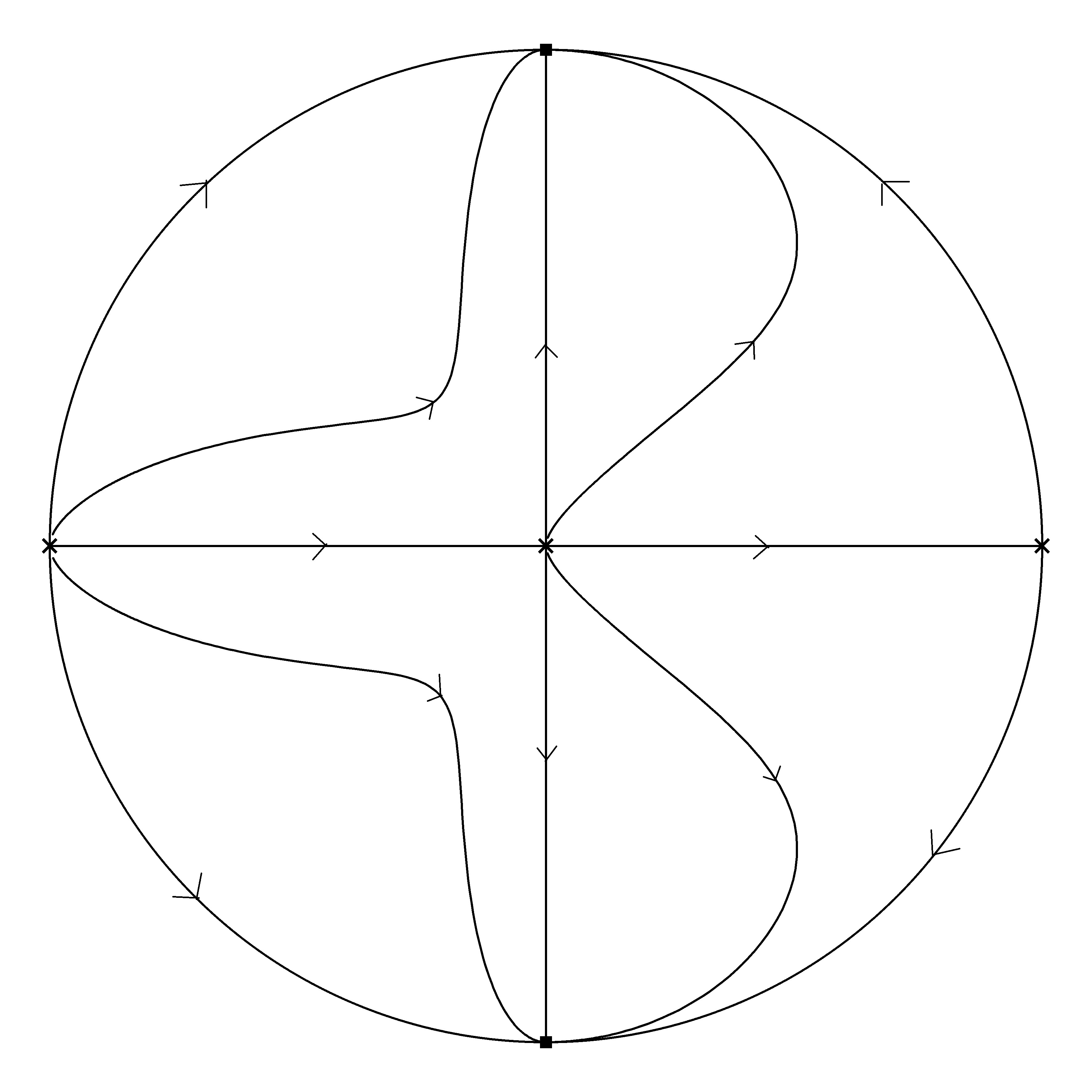}
			\caption{}
		\end{subfigure}
	\end{subfigure}

	\caption{Portraits of $(3d)$ when the infinity has 2 singularities}
	\label{fig:3top_2bottom}
\end{figure}

	\begin{figure}[H]
		\centering
		\begin{subfigure}[b]{0.3\textwidth}
			\centering
			\includegraphics[width=\textwidth]{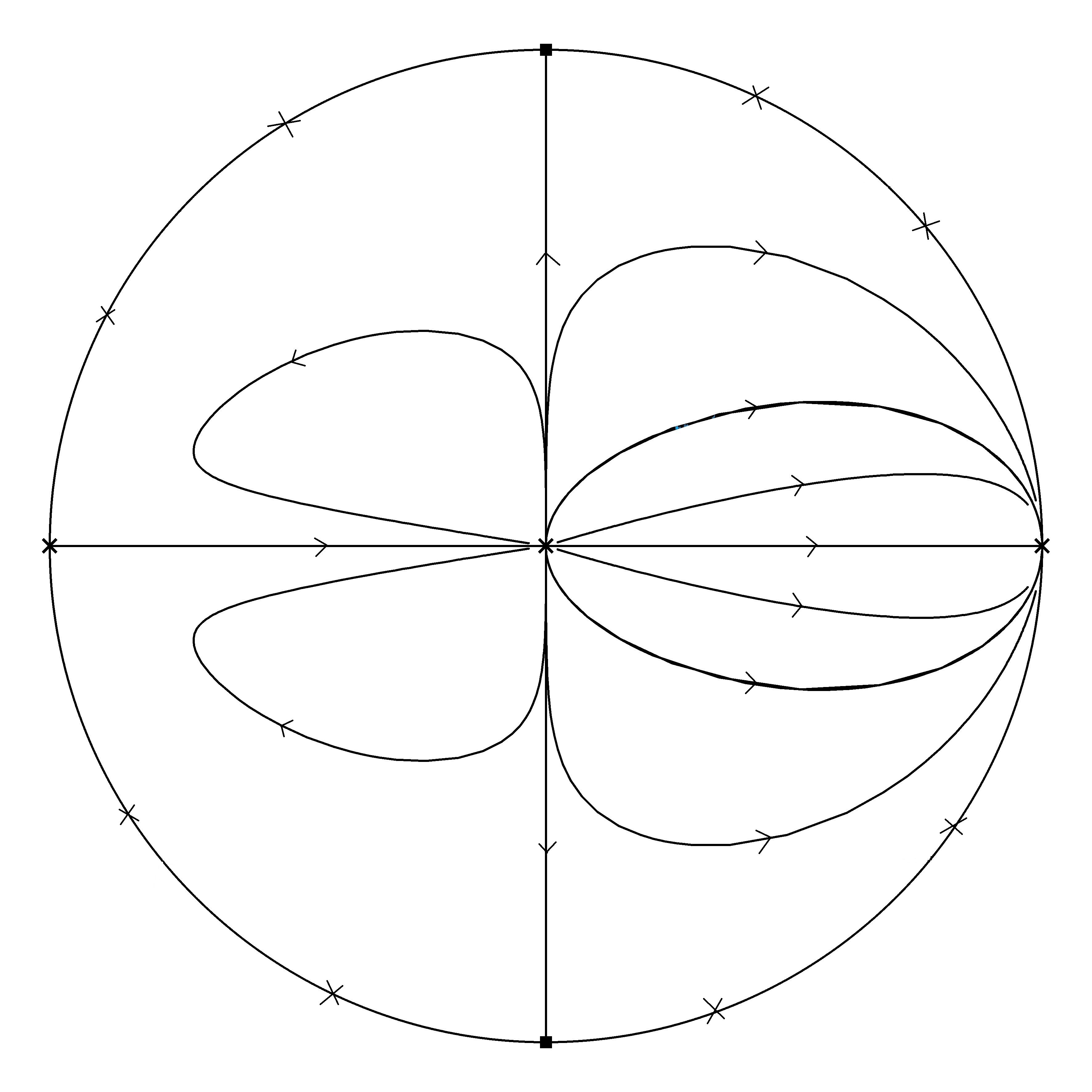}
			\caption{}
		\end{subfigure}
		\hfill
		\begin{subfigure}[b]{0.3\textwidth}
			\centering
			\includegraphics[width=\textwidth]{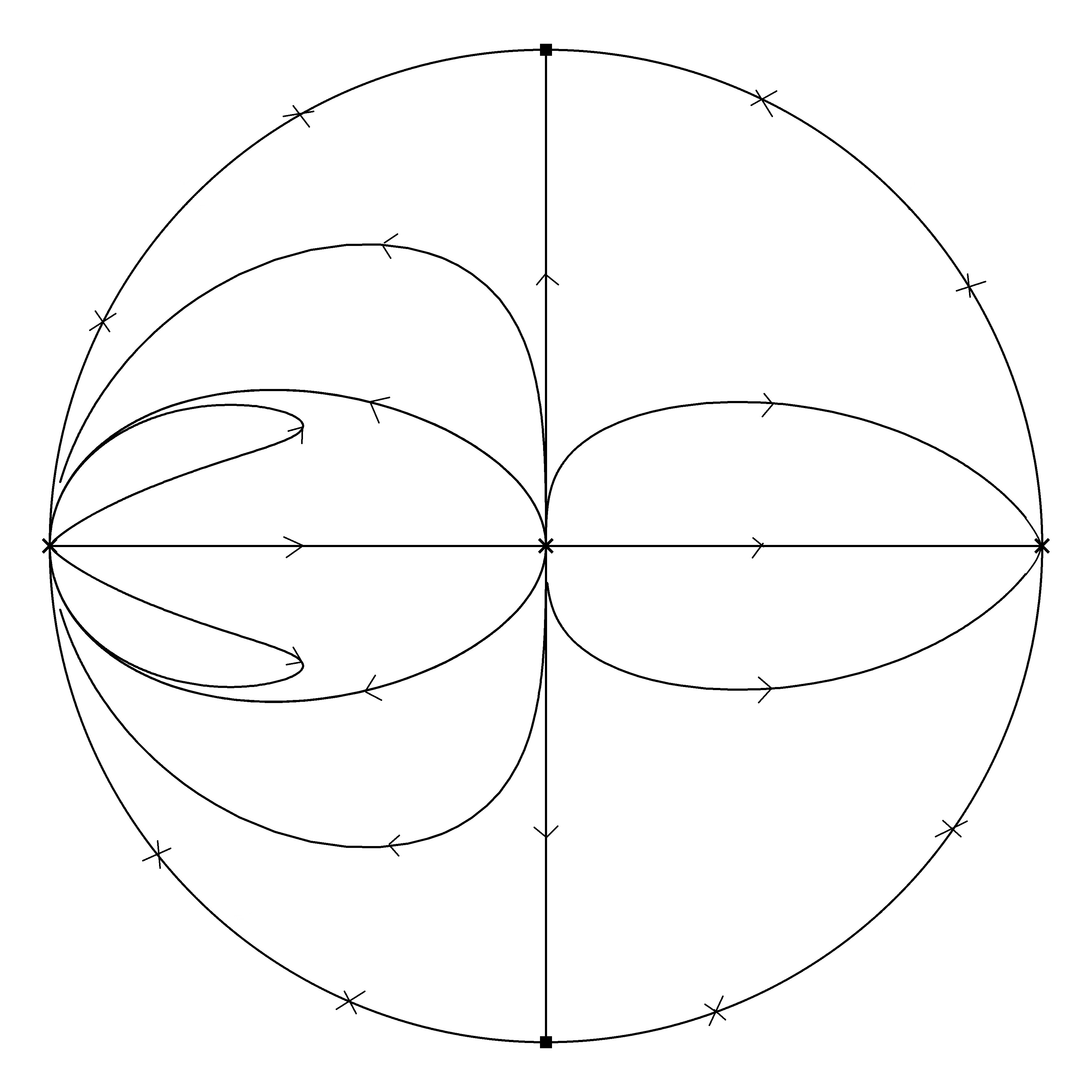}
			\caption{}
		\end{subfigure}
		\hfill
		\begin{subfigure}[b]{0.3\textwidth}
			\centering
			\includegraphics[width=\textwidth]{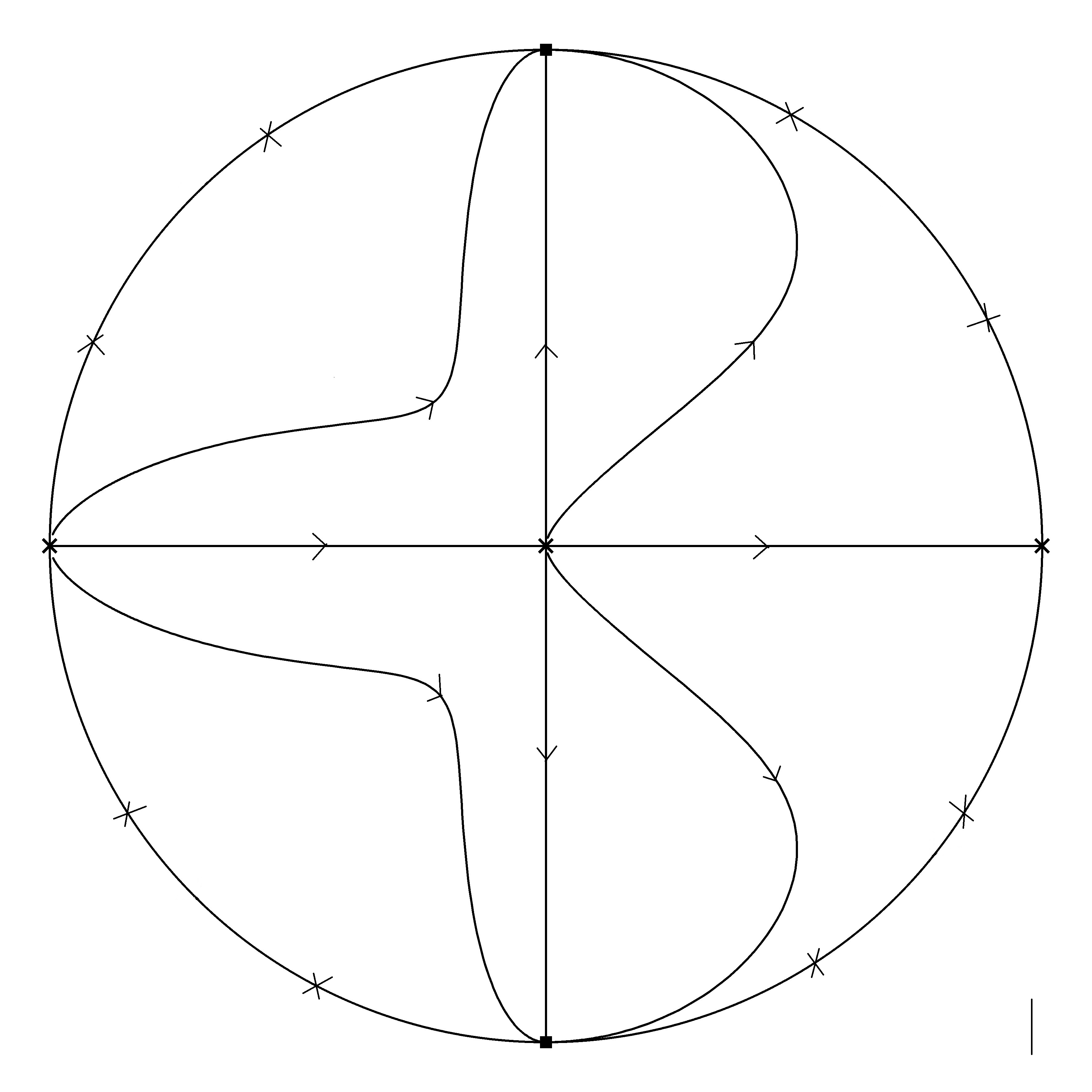}
			\caption{}
		\end{subfigure}
		\caption{Portraits of $(3d)$ when the infinity fulfills singularities}
		\label{fig:three_in_one_row}
	\end{figure}

\end{proof}
\noindent\textbf{Remark 5}\quad 
	Although certain phase portraits of quasi-homogeneous polynomial systems are topologically equivalent to those of homogeneous systems, we deliberately distinguish them in this study. This distinction is made to provide a more precise depiction of the actual global structures that arise under specific parameter conditions. It is important to note that, in the sense of topological equivalence, different types of singularities such as nodes and focus may be classified similarly, even though their geometric behaviors are different. In our analysis, the phase portraits are not only classified topologically but also derived from explicit parameter values, allowing us to simulate the systems and present their concrete geometric configurations. This approach highlights the subtle but meaningful differences that can occur within the same topological class.

For $\mathcal{H}_{1}$, it is a linear system. From the form of the system we can know that the origin can be a saddle, node, focus or a center. Here we omitted the parameter conditions and phase portraits of the homogeneous systems since it's trivial. We directly give the phase portraits of the quasi-homogeneous systems in Figure 6 without taking into account the direction of the time. Similarly, considering the symmetry and the topologically equivalence, we can always consider the condition where separatrixes can only be $x-$axis and $y-$axis for the case of the saddle. For the case of the focus, the original quasi-homogeneous polynomial system actually consists of closed tracks, which is topologically equivalent to the case of center. Hence we will give four conditions in Figure 6, where the origin is separately a saddle, node, focus or a center for the corresponding homogeneous polynomial systems.
\begin{figure}[H]
    \centering
    \begin{subfigure}[b]{0.65\textwidth}
        \centering
        \begin{subfigure}[b]{0.48\textwidth}
            \centering
            \includegraphics[width=\textwidth]{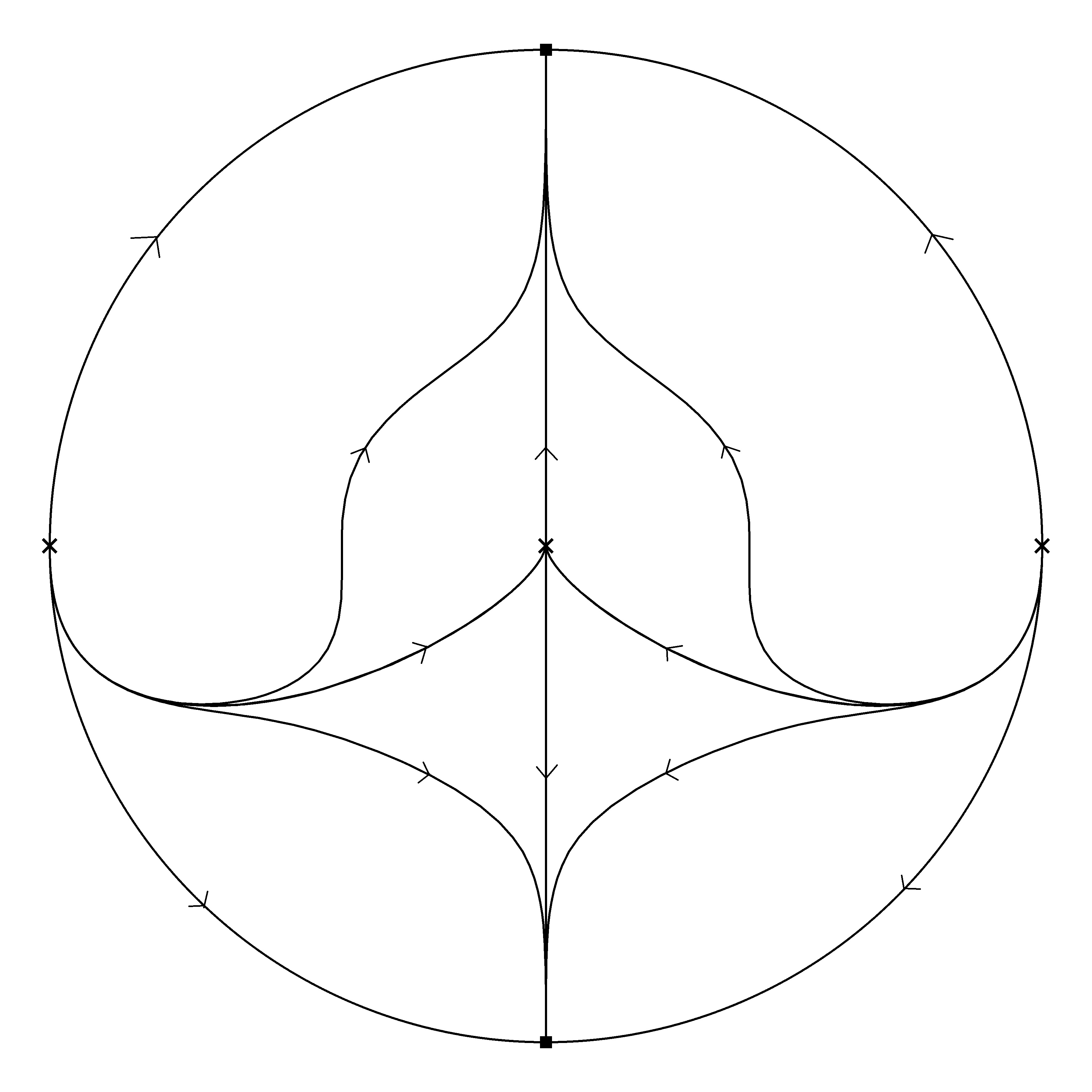}
            \caption{}
        \end{subfigure}
        \hfill
        \begin{subfigure}[b]{0.48\textwidth}
            \centering
            \includegraphics[width=\textwidth]{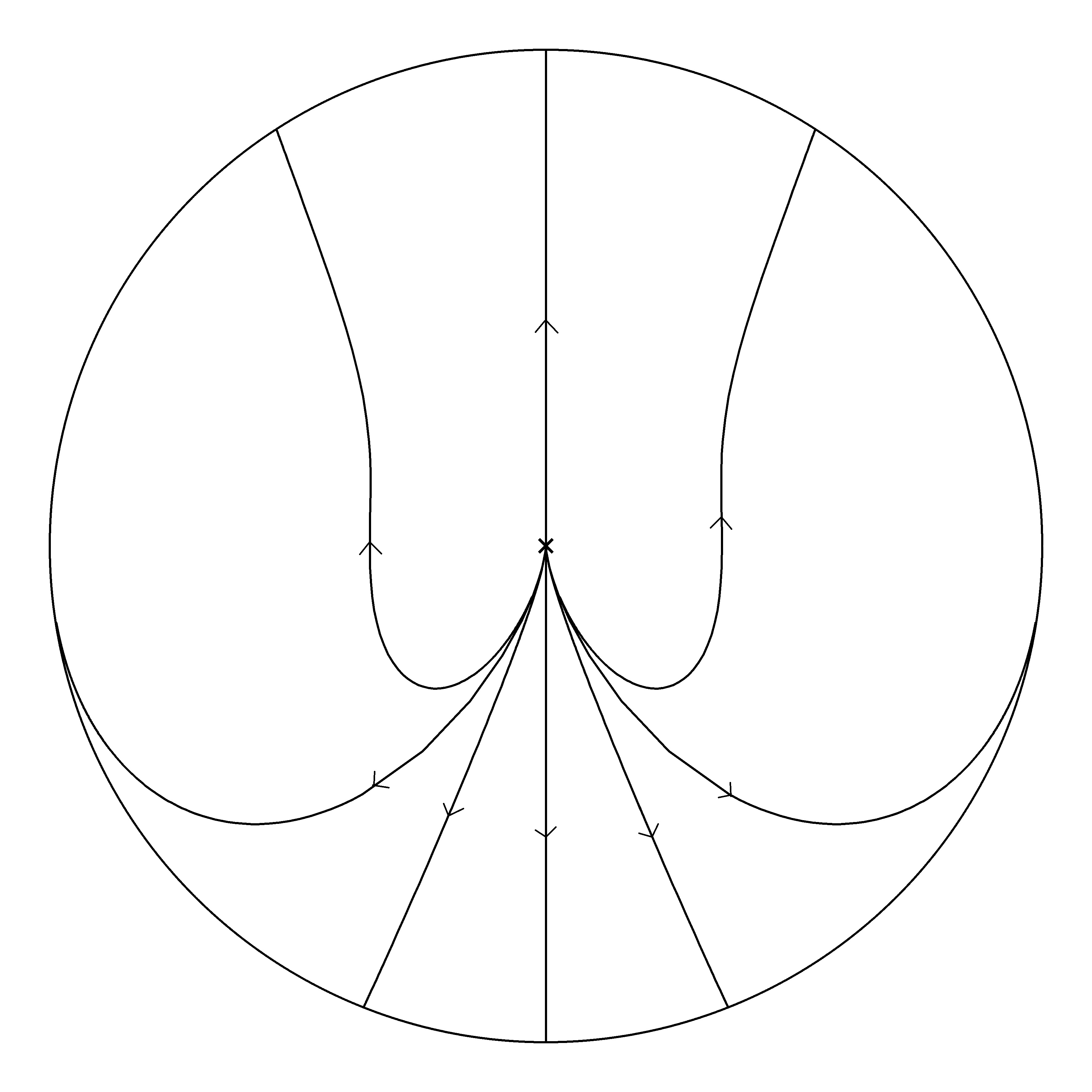}
            \caption{}
        \end{subfigure}
    \end{subfigure}

    \vspace{0.5cm} 

    \begin{subfigure}[b]{0.65\textwidth}
        \centering
        \begin{subfigure}[b]{0.48\textwidth}
            \centering
            \includegraphics[width=\textwidth]{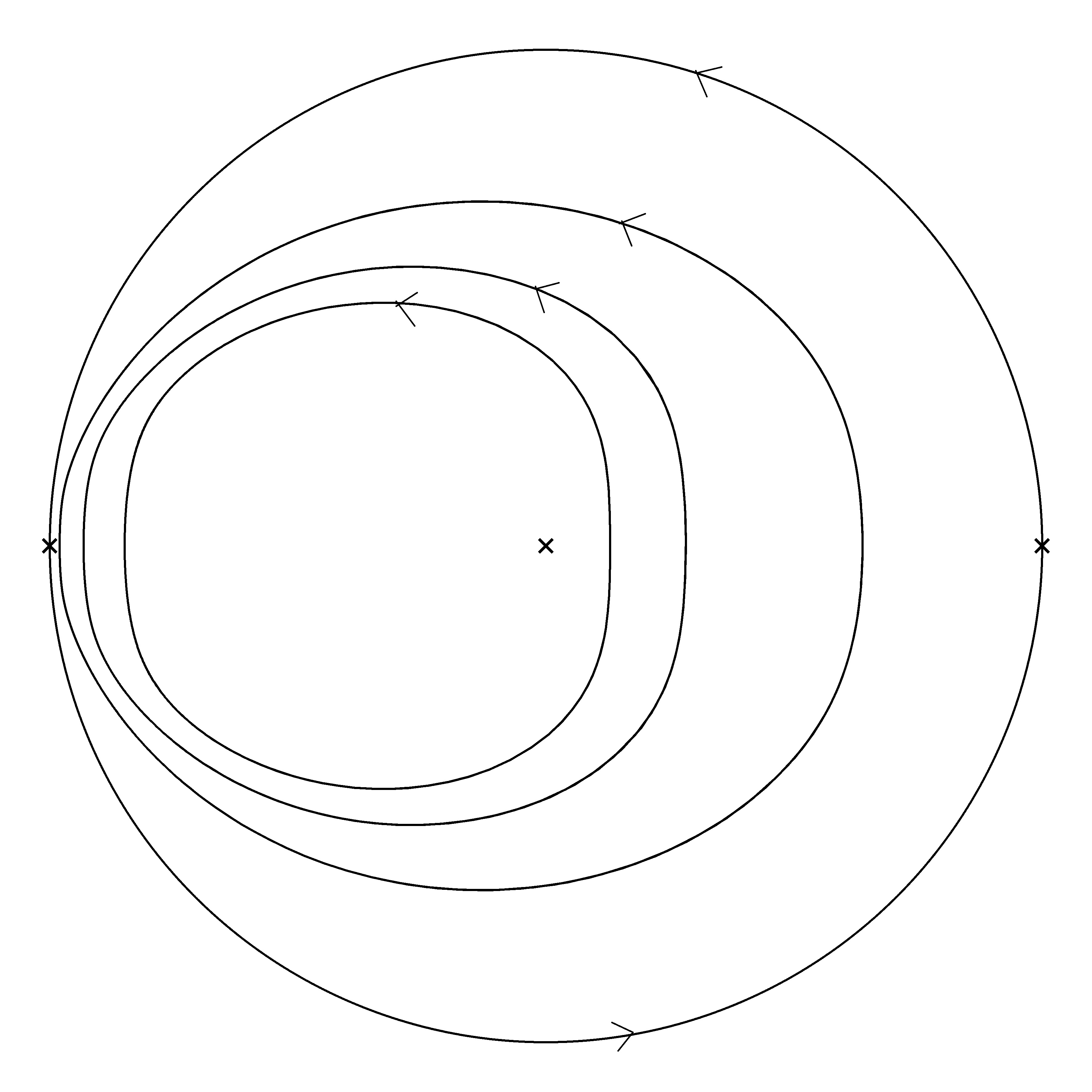}
            \caption{}
        \end{subfigure}
        \hfill
        \begin{subfigure}[b]{0.48\textwidth}
            \centering
            \includegraphics[width=\textwidth]{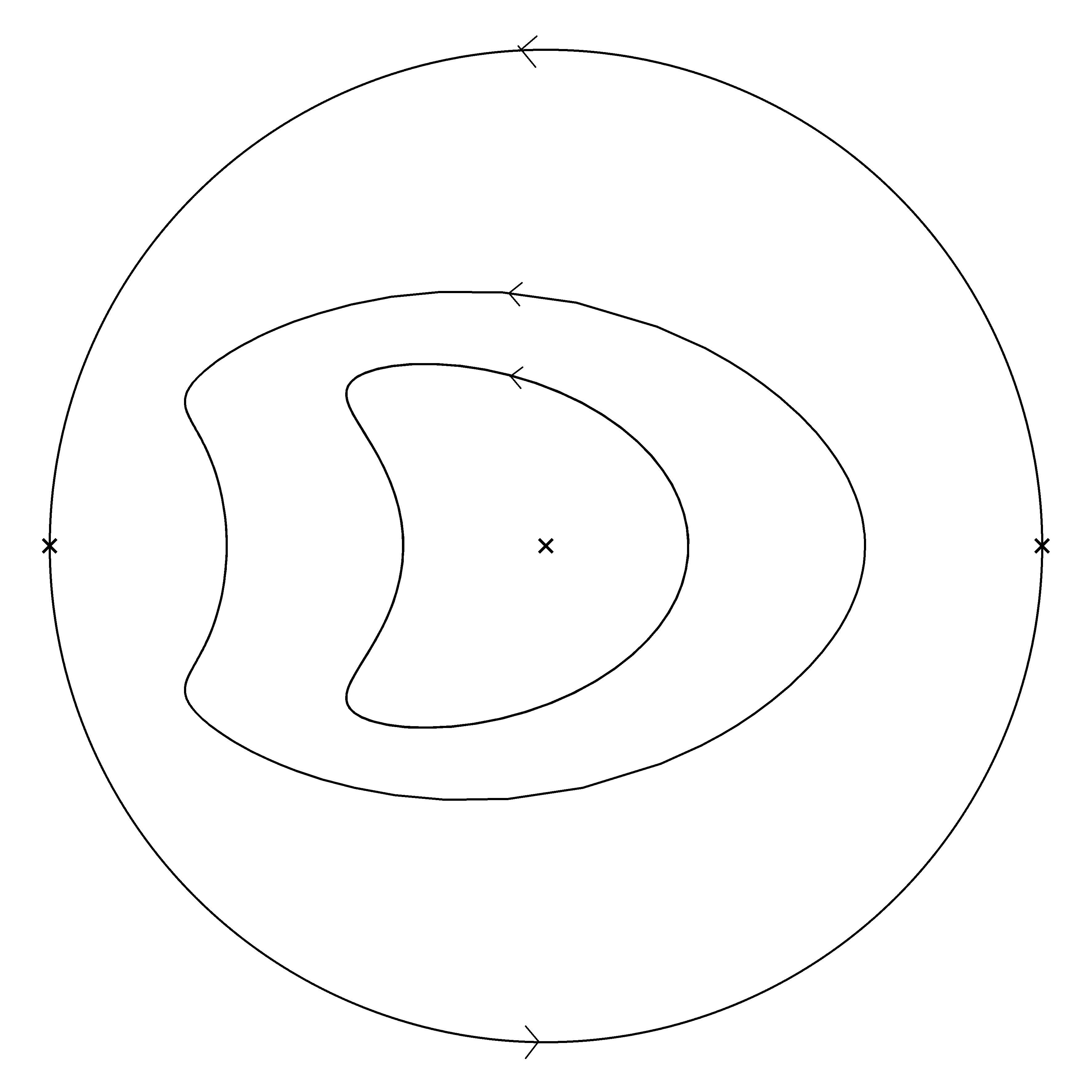}
            \caption{}
        \end{subfigure}
    \end{subfigure}

    \caption{Portraits of quasi-homogeneous systems corresponding to $\mathcal{H}_{1}$}
    \label{fig:2x2_images}
\end{figure}

 For $\mathcal{H}_{0}$, it's a constant system. The global structure is not difficult to be obtained, which is a constant vector field. The portraits is omitted here. Its corresponding quasi-homogeneous systems are (2a) and (3c). We analyze the systems directly taking the \text{Poincar\'{e}} transformations. Then they have infinity singularities located at the end of the $x-$axis. Notice that they are respectively symmetric about the $y-$axis and the $x-$axis. Then we can get their phase portraits respectively in Figure 7 (A) and (B), which are actually topologically equivalent.
 \begin{figure}[H]
    \centering
    \begin{subfigure}[b]{0.65\textwidth}
        \centering
        \begin{subfigure}[b]{0.48\textwidth}
            \centering
            \includegraphics[width=\textwidth]{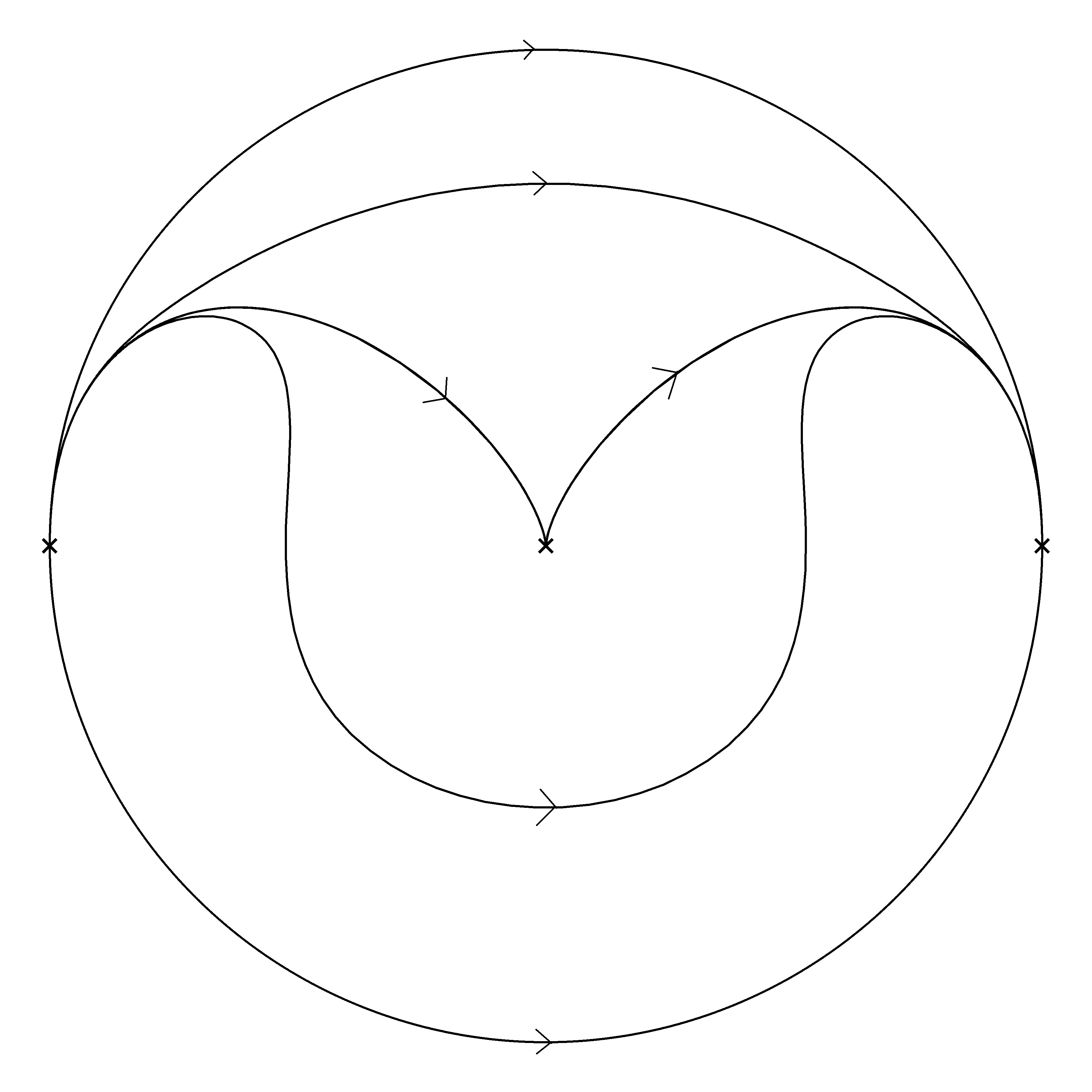}
            \caption{}
        \end{subfigure}
        \hfill
        \begin{subfigure}[b]{0.48\textwidth}
            \centering
            \includegraphics[width=\textwidth]{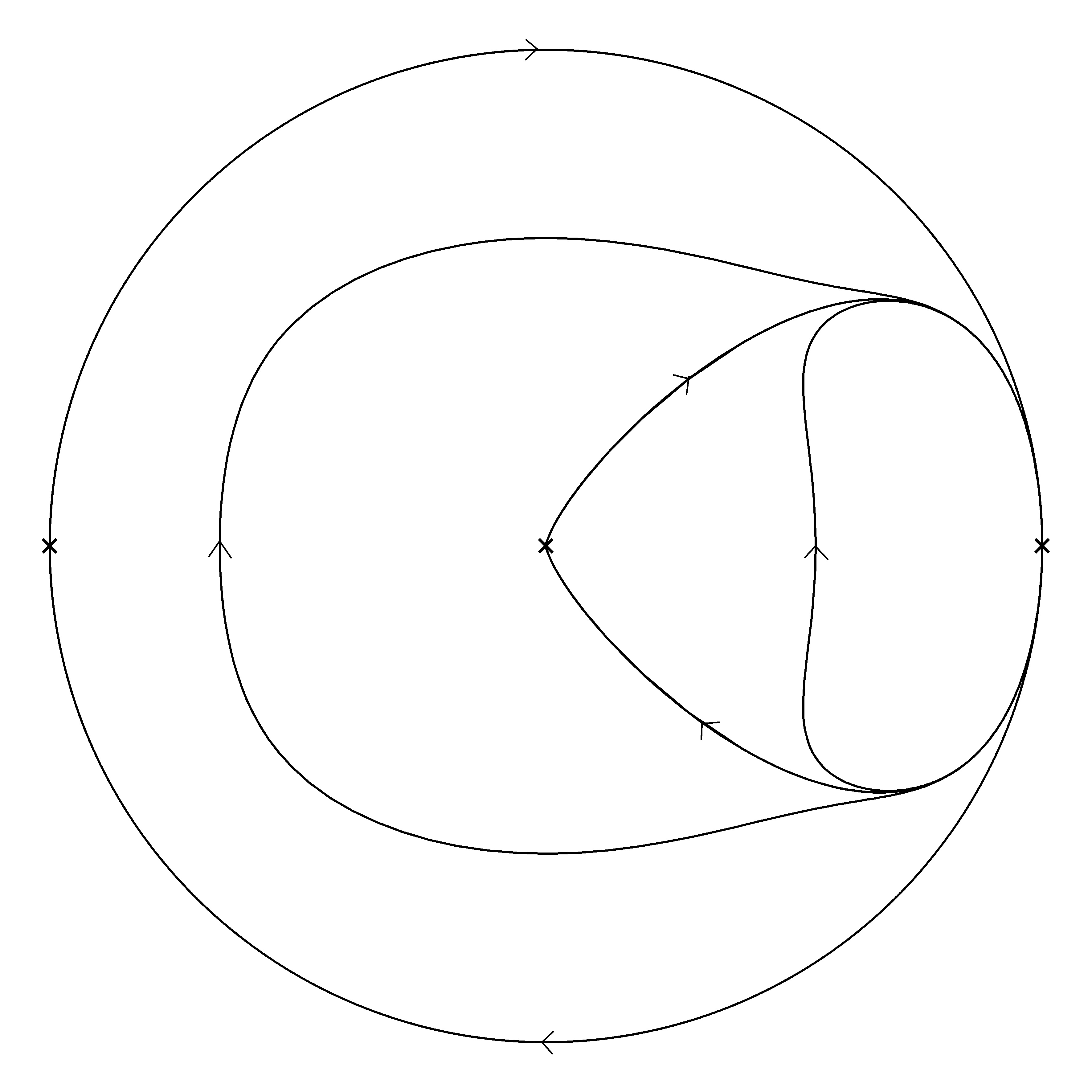}
            \caption{}
        \end{subfigure}
    \end{subfigure}
    \caption{Portraits of quasi-homogeneous systems corresponding to $\mathcal{H}_{0}$}
    \label{fig:two_side_by_side}
\end{figure}

 Finally, the paper derives the parameter conditions under which distinct global phase portraits occur, and classifies these portraits in the sense of topological equivalence. As a result, we have the following theorems.
 \begin{theorem}
	The global phase portraits of quadratic quasi-homogeneous differential systems are topologically equivalent to one of the following portraits: Figure 6 (A), (B) and Figure 7 (A) without taking into account the direction of the time.
\end{theorem}
 \begin{theorem}
	The global phase portraits of cubic quasi-homogeneous differential systems are topologically equivalent to one of the following portraits: Figure 4, Figure 5, Figure 6 and Figure 7 (B) without taking into account the direction of the time.
\end{theorem}
 

\section*{Acknowledgments}
The authors are grateful to the referees
for their careful checking and helpful comments.



\graphicspath{{figures/}}
\DeclareGraphicsExtensions{.pdf,.eps,.png,.jpg,.jpeg}

\begin{thebibliography}{79}



\bibitem{GARCIA20133185}
Belén García, Jaume Llibre, Jesús S. Pérez del Río, Planar quasi-homogeneous polynomial differential systems and their integrability, {\it Journal of Differential Equations} {\bf 255} (2013), 3185--3204.

\bibitem{ZHAO2005563}
Yulin Zhao, Zhifen Zhang, Abelian integrals and period functions for quasihomogeneous Hamiltonian vector fields, {\it Applied Mathematics Letters} {\bf 18} (2005), 563--569.

\bibitem{AZIZ2014233}
W. Aziz, J. Llibre, C. Pantazi, Centers of quasi-homogeneous polynomial differential equations of degree three, {\it Advances in Mathematics} {\bf 254} (2014), 233--250.

\bibitem{TANG201990}
Yilei Tang, Xiang Zhang, Global dynamics of planar quasi-homogeneous differential systems, {\it Nonlinear Analysis: Real World Applications} {\bf 49} (2019), 90--110.

\bibitem{MARTINEZ20165923}
Y.P. Martínez, C. Vidal, Classification of global phase portraits and bifurcation diagrams of Hamiltonian systems with rational potential, {\it Journal of Differential Equations} {\bf 261} (2016), 5923--5948.

\bibitem{Dumortier2007}
Freddy Dumortier, Jaume Llibre, Joan Artés, Qualitative Theory Of Planar Differential Systems, {\it Qualitative Theory of Planar Differential Systems} (2007).

\bibitem{CIMA1990420}
Anna Cima, Jaume Llibre, Algebraic and topological classification of the homogeneous cubic vector fields in the plane, {\it Journal of Mathematical Analysis and Applications} {\bf 147} (1990), 420--448.

\bibitem{zhang1985}
Zhifeng Zhang, Wen-Zao Huang, and Zhenxi Dong,
\textit{Qualitative Theory of Differential Equations},
Modern Mathematics Foundation Series: Collector's Edition,
Science Press, Beijing, 1985.


\bibitem{arnold1990ten}
Vladimir I. Arnold, Ten Problems, {\it Theory of Singularities and Its Applications} {\bf 1} (1990), 1--8.

\bibitem{Date1979}
T. Date, Classification and analysis of two-dimensional homogeneous quadratic differential equations systems, {\it J. Differential Equations} {\bf 32} (1979), 311--334.

\bibitem{Newton1978}
T.A. Newton, Two dimensional homogeneous quadratic differential systems, {\it SIAM Rev.} {\bf 20} (1978), 120--138.

\bibitem{Sibirskii1977}
K.S. Sibirskii, N.I. Vulpe, Geometric classification of quadratic differential systems, {\it Differential Equations} {\bf 13} (1977), 548--556.

\bibitem{Vdovina1984}
E.V. Vdovina, Classification of singular points of the equation $y' = \frac{a_0x^2 + a_1xy + a_2y^2}{b_0x^2 + b_1xy + b_2y^2}$
  by Forster's method, {\it Differ. Uravn.} {\bf 20} (1984), 1809--1813.

\bibitem{Cima1990}
A. Cima, J. Llibre, Algebraic and topological classification of the homogeneous cubic systems in the plane, {\it J. Math. Anal. Appl.} {\bf 147} (1990), 420--448.

\bibitem{Ye1995}
Y.Q. Ye, Qualitative Theory of Polynomial Differential Systems, {\it Shanghai Science and Technology Pub.} (1995).

\bibitem{Llibre1996}
J. Llibre, J.S. Pérez del Río, J.A. Rodríguez, Structural stability of planar homogeneous polynomial vector fields. Applications to critical points and to infinity, {\it J. Differential Equations} {\bf 125} (1996), 490--520.

\bibitem{Algaba2009}
A. Algaba, E. Gamero, C. García, The integrability problem for a class of planar systems, {\it Nonlinearity} {\bf 22} (2009), 396--420.

\bibitem{Gine2013}
J. Giné, M. Grau, J. Llibre, Polynomial and rational first integrals for planar quasi-homogeneous polynomial differential systems, {\it Discrete Contin. Dyn. Syst.} {\bf 33} (2013), 4531--4547.

\bibitem{Goriely1996}
A. Goriely, Integrability, partial integrability, and nonintegrability for systems of ordinary differential equations, {\it J. Math. Phys.} {\bf 37} (1996), 1871--1893.

\bibitem{Hu2007}
Y. Hu, On the integrability of quasihomogeneous systems and quasidegenerate infinity systems, {\it Adv. Differential Equations} {\bf 2007} (2007), 98427.

\bibitem{Garcia2003}
I. García, On the integrability of quasihomogeneous and related planar vector fields, {\it Int. J. Bifurcation Chaos} {\bf 13} (2003), 995--1002.

\bibitem{Garcia2013}
B. García, J. Llibre, J.S. Pérez del Río, Planar quasihomogeneous polynomial differential systems and their integrability, {\it Journal of Differential Equations} {\bf 255} (2013), 3185--3204.

\bibitem{Li2009}
W. Li, J. Llibre, J. Yang, Z. Zhang, Limit cycles bifurcating from the period annulus of quasi–homogeneous centers, {\it J. Dynam. Differential Equations} {\bf 21} (2009), 133--152.

\bibitem{Zhang2018}
Liwei Zhang, Jiang Yu, On the center criterion of planar quasi-homogeneous polynomial differential systems, {\it Bulletin des Sciences Mathématiques} {\bf 147} (2018), 7--25.




\end{thebibliography}
\end{document}